\documentclass[onecolumn, 12pt]{IEEEtran}
\usepackage{amsmath,amsfonts,amssymb,euscript, graphicx, 
epsfig,
enumerate,float,afterpage, subfigure, ifthen}%

\newtheorem{thm}{Theorem}
\newtheorem{cor}{Corollary}
\newtheorem{lem}{Lemma}

\newcommand{\expect}[1]{\mathbb{E}\left\{#1\right\}}
\newcommand{\defequiv}{\mbox{\raisebox{-.3ex}{$\overset{\vartriangle}{=}$}}}

\newcommand{\norm}[1]{||{#1}||}

\newcommand{\bv}[1]{{\boldsymbol{#1} }}
\newcommand{\script}[1]{{{\cal{#1} }}}

\begin{document}

\title{Queue Stability and Probability 1 Convergence via Lyapunov Optimization}
\author{Michael J. Neely%
\thanks{Michael J. Neely is with the  Electrical Engineering department at the University
of Southern California, Los Angeles, CA. (web: http://www-rcf.usc.edu/$\sim$mjneely).} 
\thanks{This material is supported in part  by one or more of 
the following: the DARPA IT-MANET program
grant W911NF-07-0028, 
the NSF Career grant CCF-0747525, and continuing through participation in the 
Network Science Collaborative Technology Alliance sponsored
by the U.S. Army Research Laboratory.}}

\markboth{}{Neely}

\maketitle

\begin{abstract} 
Lyapunov drift and Lyapunov optimization are powerful techniques for optimizing time averages in 
stochastic queueing networks subject to stability.   However, there are various definitions
of queue stability in the literature, and the most convenient Lyapunov drift conditions
often provide stability and performance bounds only in terms of a time average expectation, 
rather than a pure time average. We extend the theory to show that for quadratic Lyapunov functions, 
 the basic drift condition, together with a mild bounded fourth moment condition,  
 implies all major forms of stability.  Further, we show that the basic 
drift-plus-penalty condition implies that the same bounds for queue backlog and penalty 
expenditure that are known to hold for time average expectations also hold for pure time averages with 
probability 1.   Our analysis combines Lyapunov drift theory with the Kolmogorov law of large numbers
for martingale differences with finite variance. 
\end{abstract} 


\section{Introduction} 

\nocite{now}\nocite{neely-thesis}\nocite{neely-fairness-infocom05}\nocite{neely-energy-it}\nocite{neely-asilomar08}\nocite{stolyar-greedy}\nocite{atilla-fairness-ton}\nocite{neely-utility-delay-jsac}\nocite{neely-energy-delay-it}\nocite{tutorial-lin}\nocite{atilla-primal-dual-jsac}\nocite{chiang-sno}\nocite{primal-dual-cmu}

Lyapunov optimization is a powerful technique for optimizing time averages in stochastic queueing 
networks (see \cite{now}-\cite{primal-dual-cmu}). 
Work in \cite{now} presents a \emph{drift-plus-penalty theorem} 
that provides a methodology for designing control algorithms to maximize time average network utility subject to queue
stability.  The theorem also provides explicit performance tradeoffs between utility maximization and average queue
backlog. 
Example applications include maximizing network throughput subject to average power constraints, 
minimizing average power expenditure subject to network stability, and maximizing network
throughput-utility subject to network stability \cite{now}-\cite{neely-asilomar08}. 
The drift-plus-penalty theorem provides performance bounds in terms of time average expectations. 
  Time average expectations are
the same as pure time averages (with probability 1) in certain cases, such as when 
the system evolves on an irreducible and positive recurrent Markov chain with a finite or countably infinite state space (and
when some additional mild assumptions are satisfied). 
However, many systems have an uncountably infinite state space and/or do not have the required Markov structure. It is 
not clear if pure time averages satisfy the same guarantees in general.  This paper proves a
sample path version of the 
drift-plus-penalty theorem, showing that if  \emph{fourth moment boundedness} conditions are satisfied, then 
the same guarantees hold for  pure time averages with probability 1.  

To understand this result and the systems it can be applied to, 
we consider a stochastic queueing network that evolves in discrete time with unit timeslots $t \in \{0, 1, 2,\ldots\}$. 
Suppose there are $K$ queues, and let $\bv{Q}(t) = (Q_1(t), \ldots, Q_K(t))$ represent the vector of current queue backlogs. 
Random events, such as random channel conditions and packet arrivals, can take place every slot.  A network
controller reacts to the random events by choosing a control action every slot.  The control action affects queue arrival
and service variables on slot $t$, and also incurs a vector of \emph{penalties} $\bv{y}(t) = (y_0(t), y_1(t), \ldots, y_M(t))$.  
The goal is to stabilize all network queues
while minimizing the time average of $y_0(t)$ subject to the time 
averages of $y_m(t)$ being less than or equal to 0: 
\begin{eqnarray}
\mbox{Minimize:} && \overline{y}_0 \label{eq:foofoo1} \\
\mbox{Subject to:} &(1)& \overline{y}_m \leq 0 \: \: \forall m \in \{1, \ldots, M\}  \label{eq:foofoo2} \\
 &(2)& \mbox{All queues are stable}  \label{eq:foofoo3} 
\end{eqnarray} 

Assuming that the problem is feasible and that a certain drift-plus-penalty condition is met, 
the existing drift-plus-penalty theory in \cite{now} can solve this problem by 
specifying a class of algorithms, parameterized by a constant $V\geq 0$ chosen as desired,  to yield: 
\begin{eqnarray}
\limsup_{t\rightarrow\infty} \frac{1}{t}\sum_{\tau=0}^{t-1} \expect{y_0(\tau)} &\leq& y_0^* + O(1/V) \label{eq:intro1} \\
\limsup_{t\rightarrow\infty} \frac{1}{t}\sum_{\tau=0}^{t-1} \expect{y_m(\tau)} &\leq& 0 \: \: \forall m \in \{1, \ldots, M\} \label{eq:intro2} \\
\limsup_{t\rightarrow\infty} \frac{1}{t}\sum_{\tau=0}^{t-1} \sum_{k=1}^K\expect{|Q_k(\tau)|} &\leq& O(V) \: \: \forall k \in \{1, \ldots, K\}  \label{eq:intro3} 
\end{eqnarray}
where $y_0^*$ is the infimum time average of $y_0(t)$ over all algorithms that can satisfy the desired constraints. 
The guarantee (\ref{eq:intro3}) implies that the $\limsup$ time average expected queue backlog is finite for all queues, and is
a condition often called \emph{strong stability}.  The above bounds say that the time average constraints $\overline{y}_m \leq 0$
are satisfied for all $m \in \{1, \ldots, M\}$ in a \emph{time average expected sense}, that all queues $Q_k(t)$ 
are strongly stable with average 
backlog $O(V)$, and time average expected penalty is within $O(1/V)$ of optimality.  The $O(1/V)$ penalty gap can be made arbitrarily 
small by choosing a suitably large $V$ parameter, at the expense of increasing the average backlog bound linearly with $V$. 

We would like to know  when we can also claim that: 
\begin{eqnarray}
\limsup_{t\rightarrow\infty} \frac{1}{t}\sum_{\tau=0}^{t-1} y_0(\tau) &\leq& y_0^* + O(1/V) \: \: (w.p.1) \label{eq:intro1a} \\
\limsup_{t\rightarrow\infty} \frac{1}{t}\sum_{\tau=0}^{t-1} y_m(\tau) &\leq& 0 \: \: (w.p.1) \: \: \: \forall m \in \{1, \ldots, M\}  \label{eq:intro2b} \\
\limsup_{t\rightarrow\infty} \frac{1}{t}\sum_{\tau=0}^{t-1} \sum_{k=1}^K|Q_k(\tau)| &\leq& O(V) \: \: (w.p.1) \: \: \: \forall k \in \{1, \ldots, K\}\label{eq:intro3c} 
\end{eqnarray}
where ``w.p.1'' stands for ``with probability 1.'' 
This paper shows that (\ref{eq:intro1a})-(\ref{eq:intro3c}) hold
if a similar drift-plus-penalty condition holds, and additionally if the $y_m(t)$ penalties and the one-slot
changes in queue backlogs
have conditionally bounded fourth 
moments given the past. 

We note that related problems of minimizing convex functions of time averages, rather than minimizing time 
averages themselves, can be transformed into problems of the type (\ref{eq:foofoo1})-(\ref{eq:foofoo3})
using a technique of \emph{auxiliary
variables}  \cite{neely-fairness-infocom05}\cite{now}\cite{neely-utility-delay-jsac}\cite{neely-mwl-arxiv}.  Hence, these extended
problems can also be treated using the framework of this paper. However,  for brevity we limit attention
to problems of the type (\ref{eq:foofoo1})-(\ref{eq:foofoo3}).

\subsection{On relationships between time average expectations and time averages} 

It is known by Fatou's Lemma 
that if a random process is deterministically lower-bounded (such as being non-negative)
and has time averages that converge to a constant with probability 1, then this
constant must be less than or equal to the $\liminf$ time average expectation \cite{williams-martingale}.   
Thus, the inequalities (\ref{eq:intro1})-(\ref{eq:intro3})
imply (\ref{eq:intro1a})-(\ref{eq:intro3c}) when the $y_m(t)$ and $|Q_k(t)|$ processes 
are deterministically lower bounded and have convergent
time averages with probability 1.  Systems that evolve on 
positive recurrent  irreducible Markov chains with finite or countably infinite state space
can often be shown to have convergent
time average penalties. 
Further,  if the Markov chain is irreducible and has a finite or countably infinite state space
with the property that the event $\{\sum_{k=1}^K |Q_k| > \theta\}$ corresponds to 
only a finite number of states for each real number $\theta$, then the condition (\ref{eq:intro3}) implies positive recurrence. 
However, in addition to the actual network queues, the drift-plus-penalty method introduces \emph{virtual queues} 
to enforce the desired time average constraints.  These
queues typically give the overall system an 
uncountably infinite state space. 
Time average convergence
can be shown using generalized Harris recurrence theory for Markov chains with uncountably infinite state space, 
provided that certain \emph{generalized irreducibility} 
assumptions and \emph{petite set} assumptions are satisfied \cite{meyn-book}.  However, it is often difficult to check if these assumptions
hold for the particular systems of interest. 

Strong stability of a queue $Q(t)$, 
together
with either deterministically bounded arrival or server rate processes, implies \emph{rate stability} \cite{stability}.  Rate stability of $Q(t)$
means that $\lim_{t\rightarrow\infty} Q(t)/t = 0$ with probability 1.   This 
result can be used to prove that (\ref{eq:intro2b}) holds if the $y_m(t)$
processes are suitably deterministically bounded on each slot $t$.   However, this does not ensure the constraints (\ref{eq:intro1a}) or (\ref{eq:intro3c})
hold.

Certain types of systems, such as networks with flow control, often have a structure that 
yields \emph{deterministically bounded} 
queues \cite{neely-energy-it}\cite{neely-universal-scheduling-cdc2010}, which can be used to 
ensure constraints (\ref{eq:intro2b})-(\ref{eq:intro3c}) hold for those systems.  However, this requires special structure, and 
it also does not ensure (\ref{eq:intro1a}) holds unless suitable Markov chain assumptions are met.  

\subsection{Alternative algorithms} 
 

A dual-based algorithm related to the drift-plus-penalty method is considered for a wireless downlink with ``infinite
backlog''  in \cite{atilla-fairness-ton}, and convergence to near-optimal utility is shown using a countable state space
Markov chain assumption.  Stochastic approximation algorithms are used in \cite{lee-stochastic-scheduling}, 
and diminishing stepsize convex programming is used in  \cite{lin-shroff-cdc04} to treat problems that are 
more deterministic in structure. 
The works \cite{atilla-fairness-ton}\cite{lee-stochastic-scheduling}\cite{lin-shroff-cdc04} do not 
show the $[O(1/V), O(V)]$ performance-backlog tradeoff. 

Primal-dual algorithms are considered for scheduling in wireless systems with ``infinite backlog'' in 
\cite{vijay-allerton02}\cite{prop-fair-down} and shown to converge to a utility optimal operating 
point,  although this work does not consider queueing or time average constraints.  
A related primal-dual algorithm is treated in \cite{stolyar-greedy} for systems with queues. A fluid version of the 
system is shown to have a utility-optimal trajectory, and it is conjectured that the actual system has a near-optimal
utility.  Recent work in  \cite{primal-dual-cmu} considers fluid analysis of 
primal-dual updates 
and proves near-optimal utility of 
the actual system with probability 1.  It also treats a more general class of objective functions that have 
time varying parameters. 
However, it considers only rate stability for queues and does not specify the $[O(1/V), O(V)]$ tradeoff.  
Work in \cite{neely-non-convex} considers stochastic queues with a non-convex objective function, and shows that
\emph{if} the throughput vector converges, it converges to a near-local optimum or a
critical point with a $[O(1/V), O(V)]$ utility-delay
tradeoff (where a near-local optimum is a near-global optimum in the special case of convex problems).

\subsection{Paper Outline} 

In the next section we review the basic drift-plus-penalty theorem and discuss the performance
bounds it provides, which are in terms of time average expectations.  We then state the main
theorem of this paper, which shows the same bounds hold as pure time averages with probability 1.  A key 
special case of this theorem is that if a certain quadratic Lyapunov drift condition is satisfied, then the network
queues satisfy all of the six major forms of queue stability.   Section \ref{section:time-averages} provides background
on the Kolmogorov law of large numbers needed in our analysis, and derives a simple but useful generalized drift-plus-penalty
theorem.  Section \ref{section:main} shows that the conditions required for the generalized drift-plus-penalty theorem to hold
are satisfied under quadratic Lyapunov functions if certain boundedness properties hold.  Section \ref{section:application} 
uses this result in queueing networks to derive bounds of the form  (\ref{eq:intro1})-(\ref{eq:intro3}) for those systems.

\section{The Drift-Plus-Penalty Theorem} \label{section:background} 

Let $\bv{Q}(t) \defequiv (Q_1(t), Q_2(t), \ldots, Q_K(t))$ be a stochastic vector with real-valued components, and
let $p(t)$ be a real-valued stochastic process on the same probability space as $\bv{Q}(t)$.   
These processes evolve in discrete time with unit 
time slots $t \in\{0, 1, 2,\ldots\}$. 
The vector $\bv{Q}(t)$ can represent \emph{queue backlogs} in a network of $K$ queues.  The process $p(t)$ can represent a \emph{penalty process}, where $p(t)$ is a real-valued
penalty (such as power expenditure)  incurred by some control action taken by the system 
on slot $t$.    While typical queue backlogs and penalties  
are non-negative, for generality we allow them to possibly take negative values.  

For each slot $t$, define $\script{H}(t)$ as the \emph{history} of past $\bv{Q}(\tau)$  and $p(\tau)$ values, where $\bv{Q}(\tau)$ values
are taken up to and including slot $t$, and $p(\tau)$ values are taken up to but \emph{not} including slot $t$.   Specifically, define
$\script{H}(0) \defequiv \{\bv{Q}(0)\}$, and for each $t>0$ define: 
\begin{equation} \label{eq:history-Q} 
 \script{H}(t) \defequiv \{\bv{Q}(0), \bv{Q}(1), \ldots, \bv{Q}(t), p(0), p(1), \ldots, p(t-1)\} 
 \end{equation} 
 As a scalar measure of the size of the $\bv{Q}(t)$ vector, define the following \emph{quadratic Lyapunov function}: 
\begin{equation} \label{eq:lyapunov-function} 
L(\bv{Q}(t)) \defequiv \frac{1}{2}\sum_{k=1}^K w_k Q_k(t)^2 
\end{equation} 
where the constants $w_k$ are positive weights.  Define $\Delta(\script{H}(t))$ as the \emph{conditional Lyapunov
drift}: 
\begin{equation} \label{eq:lyapunov-drift} 
 \Delta(\script{H}(t)) \defequiv \expect{L(\bv{Q}(t+1)) - L(\bv{Q}(t)) |\script{H}(t)}
 \end{equation} 
 Note that $\script{H}(t)$ includes $\bv{Q}(t)$, and so the above conditional expectation is with respect
 to the conditional distribution of $\bv{Q}(t+1)$ given $\bv{Q}(0), \ldots, \bv{Q}(t), p(0), \ldots, p(t-1)$. 

 The \emph{drift-plus-penalty} algorithm
 for minimizing the time average expected penalty $p(t)$ 
 subject to queue stability operates as follows:  Every slot $t$ the network controller observes the current
 $\script{H}(t)$ and chooses a control policy that minimizes a bound on the following expression:\footnote{Strictly 
 speaking, the prior
 work in \cite{now} defines the Lyapunov drift by conditioning 
 only on $\bv{Q}(t)$, rather than on the full history $\script{H}(t)$.  We condition 
 on $\script{H}(t)$ in this paper because such conditioning is needed for application of the Kolmogorov
 law of large numbers.} 
 \begin{equation} \label{eq:dpp-expression} 
 \Delta(\script{H}(t)) + V\expect{p(t)|\script{H}(t)} 
 \end{equation} 
 where $V$ is a non-negative control parameter that is chosen to affect a desired
tradeoff between the average penalty and the average queue
backlog.   A version of this algorithm was developed in \cite{neely-thesis} for maximizing throughput-utility
subject to stability, and simple modifications were presented for other contexts in  \cite{now}\cite{neely-energy-it}. 
 This is a useful algorithm for queueing networks because it can typically 
 be implemented based only on $\bv{Q}(t)$, without keeping
 a memory of the full history and without requiring knowledge of traffic rates or channel probabilities
 (see applications in Section \ref{section:application}). 
 Such a control policy  
often gives rise to stochastic processes $\bv{Q}(t)$ and $p(t)$ that  
satisfy the following \emph{drift-plus-penalty condition} for all slots $t \in\{0, 1, 2, \ldots\}$ and all possible
$\script{H}(t)$: 
\begin{equation} \label{eq:dpp-condition} 
\Delta(\script{H}(t)) + V\expect{p(t)|\script{H}(t)} \leq B + Vp^* - \epsilon \sum_{k=1}^K|Q_k(t)| 
\end{equation} 
where $B$, $p^*$, $\epsilon$ are finite constants, with $\epsilon>0$.  The value  
$p^*$ represents a target value for the time average expectation of the 
penalty process $p(t)$.   
The following theorem from \cite{now}\cite{neely-thesis}\cite{neely-energy-it} shows that this condition immediately implies
the time average expectation of $p(t)$ is either above the target $p^*$, or is within a distance of at most $O(1/V)$ from $p^*$, 
while ensuring time average expected queue backlog is $O(V)$.    

\begin{thm} \label{thm:lyap-opt-expectation} (Lyapunov Optimization with Expectations \cite{now}\cite{neely-thesis}\cite{neely-energy-it}) Assume
that $\expect{L(\bv{Q}(0))} < \infty$, and that the condition (\ref{eq:dpp-condition}) holds for some finite 
constants $B$, $p^*$, $V>0$, and $\epsilon>0$.  If there is a finite (and possibly negative) constant $p_{min}$ such that
$\expect{p(t)} \geq p_{min}$ for all slots $t \in \{0, 1, 2, \ldots\}$, then: 
\begin{eqnarray}
\limsup_{M\rightarrow\infty} \frac{1}{M}\sum_{\tau=0}^{M-1} \expect{p(t)} &\leq& p^* + \frac{B}{V} \label{eq:expect1} \\
\limsup_{M\rightarrow\infty} \frac{1}{M}\sum_{\tau=0}^{M-1} \sum_{k=1}^K\expect{|Q_k(t)|} &\leq& \frac{B + V(p^* - p_{min})}{\epsilon}  \label{eq:expect2} 
\end{eqnarray}
Further, if (\ref{eq:dpp-condition}) holds in the case $V=0$, then inequality (\ref{eq:expect2}) still holds. Likewise, if (\ref{eq:dpp-condition}) holds in the case $\epsilon =0$, then inequality (\ref{eq:expect1}) still holds. 
\end{thm} 

The proof of Theorem \ref{thm:lyap-opt-expectation} requires only three lines and is repeated below to provide intuition: 
Taking expectations of (\ref{eq:dpp-condition}) and using the law of iterated expectations yields the following for all $t \in \{0, 1, 2, \ldots\}$: 
\[ \expect{L(\bv{Q}(t+1))} - \expect{L(\bv{Q}(t))}  + V\expect{p(t)} \leq B  + Vp^* - \epsilon\sum_{k=1}^K\expect{|Q_k(t)|} \]
Summing the above over $t \in \{0, \ldots, M-1\}$ for some integer $M>0$  and dividing by $M$ yields: 
\[ \frac{\expect{L(\bv{Q}(M))} - \expect{L(\bv{Q}(0))}}{M} + V\frac{1}{M}\sum_{t=0}^{M-1}\expect{p(t)} \leq B + Vp^* - \epsilon\frac{1}{M}\sum_{t=0}^{M-1} \sum_{k=1}^K\expect{|Q_k(t)|} \]
Rearranging terms in the above inequality and using the fact that $\expect{L(\bv{Q}(M))} \geq 0$ and $\expect{p(t)} \geq p_{min}$ for all $t$ 
immediately leads to the following two inequalities: 
\begin{eqnarray*}
\frac{1}{M}\sum_{t=0}^{M-1}\expect{p(t)} &\leq& p^* + \frac{B}{V} + \frac{\expect{L(\bv{Q}(0))}}{VM}  \\
\frac{1}{M}\sum_{t=0}^{M-1}\sum_{k=1}^K\expect{|Q_k(t)|} &\leq& \frac{B + V(p^* - p_{min})}{\epsilon} + \frac{\expect{L(\bv{Q}(0))}}{\epsilon M} 
\end{eqnarray*}
Taking a limit of the above inequalities as $M\rightarrow\infty$ yields (\ref{eq:expect1})-(\ref{eq:expect2}). 

\subsection{Main Result of This Paper} 

Theorem \ref{thm:lyap-opt-expectation}  illustrates an important tradeoff between time average expected
penalty and the resulting time average expected 
queue backlog.  However, one may wonder if the same bounds hold with probability 1 for pure time
averages (without the expectations).  To address this question, we impose the following additional 
\emph{boundedness assumptions}: 
\begin{itemize} 
 \item The second moments $\expect{p(t)^2}$ are finite for all $t \in\{0, 1, 2, \ldots\}$ and satisfy: 
 \begin{equation} \label{eq:boundedness1}
 \sum_{\tau=1}^{\infty}\frac{\expect{p(\tau)^2}}{\tau^2}< \infty
 \end{equation} 
 \item There is a finite (possibly negative) constant $p_{min}$ such that for all slots $t$ and all possible $\script{H}(t)$: 
 \begin{equation} \label{eq:b1point5} 
\expect{p(t)|\script{H}(t)}  \geq p_{min} 
 \end{equation} 
\item  There is a finite constant $D>0$ such that for all slots $t$,  all possible $\bv{Q}(t)$, and all $k \in \{1, \ldots, K\}$ the conditional
fourth moments of queue changes are bounded as follows:
\begin{eqnarray} \label{eq:boundedness2}  
\expect{(Q_k(t+1)-Q_k(t))^4|\bv{Q}(t)} \leq D 
\end{eqnarray}
\end{itemize} 

Note that condition (\ref{eq:boundedness1}) holds whenever $\expect{p(t)^2} \leq C$ for all $t$ for some finite constant $C>0$. 
 The following theorem is the main result of this
paper:

\begin{thm} \label{thm:lyap-opt} (Lyapunov Optimization with Pure Time Averages) Assume the boundedness
assumptions (\ref{eq:boundedness1})-(\ref{eq:boundedness2}) hold.  Let $L(\bv{Q}(t))$ be a quadratic Lyapunov function 
of the form (\ref{eq:lyapunov-function}), and assume  the initial queue backlog $\bv{Q}(0)$ is finite with probability 1.   
If the drift-plus-penalty 
condition (\ref{eq:dpp-condition}) is satisfied for all slots $t$ and all possible $\script{H}(t)$ (with finite
constants $B$, $p^*$, $V> 0$, $\epsilon>0$), then: 
\begin{eqnarray}
\limsup_{t\rightarrow\infty} \frac{1}{t}\sum_{\tau=0}^{t-1} p(\tau) &\leq& p^* + \frac{B}{V} \: \: (w.p.1)  \label{eq:expect11} \\
\limsup_{t\rightarrow\infty} \frac{1}{t}\sum_{\tau=0}^{t-1} \sum_{k=1}^K|Q_k(\tau)| &\leq& \frac{B + V(p^* - p_{min})}{\epsilon}  \: \: (w.p.1) \label{eq:expect22} 
\end{eqnarray}
where $(w.p.1)$ stand for ``with probablity 1.'' 
Further, for all $k \in \{1, \ldots, K\}$ we have: 
\begin{equation} \label{eq:lyap-opt-rs}
 \lim_{t\rightarrow\infty} \frac{Q_k(t)}{t} = 0 \: \: (w.p.1) 
 \end{equation} 
Finally, if (\ref{eq:dpp-condition}) holds in the case $V=0$, then inequality (\ref{eq:expect22}) and equality (\ref{eq:lyap-opt-rs})
 still hold.
\end{thm} 

A more detailed upper bound on time average queue backlog is provided in (\ref{eq:togethernow}) of the proof.

\subsection{Queue Stability} 
A special case of Theorem \ref{thm:lyap-opt} is when the fourth moment condition
(\ref{eq:boundedness2}) is satisfied and when 
the following drift condition holds for all $t$ and all $\script{H}(t)$: 
\begin{equation} \label{eq:queue-stability} 
 \Delta(\script{H}(t)) \leq B  - \epsilon\sum_{k=1}^K|Q_k(t)| 
 \end{equation} 
where $B>0$ and $\epsilon>0$.  This is a special case of (\ref{eq:dpp-condition}) with $V=0$ and $p(t) = p^* = 0$.
In this case we have that all queues $Q_k(t)$ in the system satisfy: 
\begin{eqnarray}
\limsup_{M\rightarrow\infty}\frac{1}{M}\sum_{t=0}^M \expect{|Q_k(t)|} &\leq& B/\epsilon \label{eq:stab1} \\
\limsup_{M\rightarrow\infty} \frac{1}{M}\sum_{t=0}^M |Q_k(t)| &\leq& B/\epsilon \: \: (w.p.1) \label{eq:stab2} \\
\lim_{q\rightarrow\infty} \left[\limsup_{M\rightarrow\infty}\frac{1}{M} \sum_{t=0}^{M-1}Pr[|Q_k(t)|>q] \right]&=& 0 \label{eq:stab3} \\
\lim_{q\rightarrow\infty} \left[\limsup_{M\rightarrow\infty}\frac{1}{M} \sum_{t=0}^{M-1}1\{|Q_k(t)|>q\}\right] &=& 0 \: \: (w.p.1) \label{eq:stab4} \\
\lim_{t\rightarrow\infty} \expect{|Q_k(t)|}/t  &=& 0 \label{eq:stab5} \\
\lim_{t\rightarrow\infty} Q_k(t)/t &=& 0 \: \: (w.p.1) \label{eq:stab6}
\end{eqnarray}
where $1\{|Q_k(t)|>q\}$ is an indicator function that is $1$ if $|Q_k(t)|>q$, and $0$ else. The above are 6 major forms of queue
stability. The inequality  (\ref{eq:stab1}) is often called \emph{strong stability}, and holds by Theorem \ref{thm:lyap-opt-expectation}. 
Its sample path version is inequality (\ref{eq:stab2}), and this holds by Theorem \ref{thm:lyap-opt}.  The inequality 
(\ref{eq:stab1}) can easily be used to prove (\ref{eq:stab3}) via 
the fact that $|Q_k(t)| \geq q1\{Q_k(t)>q\}$, and the same fact can easily prove 
that (\ref{eq:stab2}) implies  (\ref{eq:stab4}).  The stability definition (\ref{eq:stab5}) is called \emph{mean rate stability}, 
and does not follow from 
any of the above results, but follows from Theorem \ref{thm:rs-main} given below. 
The stability definition (\ref{eq:stab6}) is a sample path version called \emph{rate stability}, and is implied by 
Theorem \ref{thm:lyap-opt}.   Relationships between these various stability definitions are discussed in \cite{stability}. 
In summary, if  changes in queue backlogs have uniformly bounded conditional 
fourth moments (so that
(\ref{eq:boundedness2}) holds), and if the 
Lyapunov drift condition (\ref{eq:queue-stability}) 
holds for a quadratic Lyapunov function, then 
all queues in the network satisfy all of the major forms of stability. 

The following useful theorem shows that in the special case $\epsilon =0$, the condition (\ref{eq:queue-stability}) 
still implies rate stability and mean rate stability, 
regardless of whether or not conditional fourth moments are bounded. 

\begin{thm} \label{thm:rs-main}  (Rate Stability and Mean Rate Stability)  Let $L(\bv{Q}(t))$ be a quadratic Lyapunov
function of the form (\ref{eq:lyapunov-function}). 
Suppose there is a finite constant $B>0$ such that for all $\tau \in \{0, 1, 2, \ldots\}$ and all possible $\script{H}(\tau)$, 
we have:\footnote{The same results for Theorem \ref{thm:rs-main} hold if the requirement
``$\Delta(\script{H}(t)) \leq B$'' (which conditions on the
full history $\script{H}(t)$), is replaced with 
``$\expect{L(\bv{Q}(t+1)) - L(\bv{Q}(t))|\bv{Q}(t)} \leq B$'' (which conditions only on $\bv{Q}(t)$).} 
\begin{equation*} 
\Delta(\script{H}(\tau)) \leq B 
\end{equation*} 
Then: 

(a) If $\expect{L(\bv{Q}(0))} < \infty$, then $Q_k(t)$ is mean rate stable for all $k \in \{1, \ldots, K\}$.  That is: 
\[ \lim_{t\rightarrow\infty} \expect{|Q_k(t)|}/t = 0 \]

(b) If $\bv{Q}(0)$ is finite with probability 1, and if 
there is a finite constant $D>0$ such that for all $t\in\{0, 1, 2, \ldots\}$ and all $k \in \{1, \ldots, K\}$
we have:
 \[ \expect{(Q_k(t+1) - Q_k(t))^2} \leq D \]
 then $Q_k(t)$ is rate stable for all $k \in \{1, \ldots, K\}$.  That is: 
 \[ \lim_{t\rightarrow\infty} Q_k(t)/t = 0 \: \: (w.p.1) \]
\end{thm} 
\begin{proof}
See Appendix E.
\end{proof}

Theorem \ref{thm:rs-main} only requires the (unconditional) second moment of queue changes to be bounded, whereas
Theorem \ref{thm:lyap-opt} requires (conditional) fourth moments to be bounded. 

\section{Convergence of Time Averages}  \label{section:time-averages} 

This section reviews basic convergence definitions and results needed in the proof of Theorem \ref{thm:lyap-opt}.
It then develops a generalized drift-plus-penalty result for processes with a certain variance property. 

\subsection{Discussion of Convergence With Probability 1} \label{section:wp1}
Let $Y(t)$ be a real-valued stochastic process defined on $t \in \{0, 1, 2, \ldots\}$.  
To say that $Y(t)$ converges to a constant
$\alpha \in \mathbb{R}$ ``with probability 1'' (or ``almost surely''), we use the notation: 
\begin{equation} \label{eq:wp1} 
\lim_{t\rightarrow\infty} Y(t) = \alpha \: \: \: (w.p.1) 
\end{equation} 
It is well known that  (\ref{eq:wp1}) holds if and only if for all $\epsilon>0$ we have: 
\begin{equation} \label{eq:union} 
\lim_{n\rightarrow\infty} Pr[\cup_{t\geq n} \{|Y(t)-\alpha| \geq \epsilon\}] = 0 
\end{equation} 
Probabilities of the type (\ref{eq:union}) can be bounded via the union
bound: 
\begin{equation} 
 0 \leq Pr[\cup_{t\geq n} \{|Y(t)-\alpha|\geq \epsilon\}] \leq \sum_{t=n}^{\infty} Pr[|Y(t)-\alpha| \geq \epsilon] \label{eq:union2} 
 \end{equation} 
It follows that (\ref{eq:union}) holds if the  infinite sum on the right-hand-side of (\ref{eq:union2}) is the tail of a convergent series.   
Bounds on each term of the series can be obtained via the well known Chebyshev inequality:
\[ Pr[|Y(t) - \alpha|\geq \epsilon] \leq \frac{\expect{(Y(t)-\alpha)^2}}{\epsilon^2} \]
The above discussion explains the following well known lemma: 
\begin{lem} \label{lem:wp1} 
If $Y(t)$ satisfies the following: 
\[ \sum_{t=1}^{\infty} \expect{(Y(t)-\alpha)^2} < \infty \]
then (\ref{eq:wp1}) holds, that is, the variables $Y(t)$ converge to $\alpha$ with probability 1. 
\end{lem} 

\begin{cor} \label{cor:rate-stable}  (Rate Stability in Queues with Finite Variance) If $Q(t)$ is a real-valued
stochastic process defined over slots $t \in \{0, 1, 2, \ldots\}$ that satisfies: 
\[ \sum_{t=1}^{\infty} \frac{\expect{Q(t)^2}}{t^2} < \infty \]
then: 
\[\lim_{t\rightarrow\infty} \frac{Q(t)}{t} = 0 \: \: (w.p.1) \]
In particular, this holds whenever there is a finite constant $C>0$ such that $\expect{Q(t)^2} \leq C$ for all $t$.  
\end{cor} 
\begin{proof} 
This corollary follows as an immediate consequence of Lemma \ref{lem:wp1} by defining 
$Y(t) \defequiv Q(t)^2/t^2$ and $\alpha = 0$.   The special case when $\expect{Q(t)^2} \leq C$
follows because $\sum_{t=1}^{\infty} \frac{C}{t^2} < \infty$. 
\end{proof} 

\subsection{Time Averages and the Kolmogorov Strong Law for Martingale Differences}

Let 
$X(t)$ be a real-valued stochastic process 
defined over timeslots $t\in\{0, 1, 2, \ldots\}$.   Define the \emph{history} $\script{H}_X(t)$
to be the set of values of the process before slot $t$, so that $\script{H}_X(0)$ is the empty set, and for all slots $t>0$ we have: 
\begin{equation} \label{eq:history} 
\script{H}_X(t) \defequiv \{X(0), X(1), \ldots, X(t-1)\} 
\end{equation}
We first assume the process $X(t)$ has the property $\expect{X(t)|\script{H}_X(t)} = 0$ for all $t$ and all possible
$\script{H}_X(t)$.  Such processes are called \emph{martingale differences}. 
The following theorem is 
a well known variation on the Kolmogorov strong law of large numbers.

\begin{thm} \label{lem:martingale} (Kolmogorov strong law for martingale differences \cite{williams-martingale}\cite{olav-martingale}\cite{mart-approx-filtration}) 
Suppose that $X(t)$ is a stochastic process over $t \in \{0, 1, 2, \ldots\}$ such that: 
\begin{itemize} 
\item $\expect{X(t)|\script{H}_X(t)} = 0$ for all $t$ and all $\script{H}_X(t)$, where $\script{H}_X(t)$ is 
defined in (\ref{eq:history}). 
\item The second moments $\expect{X(t)^2}$ are finite for all $t$ and satisfy: 
\begin{equation} \label{eq:vc2}  
 \sum_{t=1}^{\infty} \frac{\expect{X(t)^2}}{t^2} < \infty 
 \end{equation} 
\end{itemize} 
Then: 
\[ \lim_{t\rightarrow\infty} \frac{1}{t}\sum_{\tau=0}^{t-1} X(\tau) = 0 \: \: \: (w.p.1) \] 
\end{thm} 

The following corollary follows easily from the Kolmogorov strong law given above. 

\begin{cor} \label{lem:supermartingale}
Let $X(t)$ be a stochastic process defined over slots $t \in \{0, 1, 2, \ldots\}$, and suppose that:  
\begin{itemize} 
\item There is a finite constant $B$ such that 
$\expect{X(t)|\script{H}_X(t)} \leq B$ for all $t$ and all $\script{H}_X(t)$, where the history $\script{H}_X(t)$ is
defined in (\ref{eq:history}). 
\item The second moments $\expect{X(t)^2}$ are finite for all $t$ and satisfy: 
\begin{equation*} 
 \sum_{t=1}^{\infty} \frac{\expect{X(t)^2}}{t^2} < \infty 
 \end{equation*} 
\end{itemize} 
Then: 
\[ \limsup_{t\rightarrow\infty} \frac{1}{t}\sum_{\tau=0}^{t-1} X(\tau) \leq B \: \: \: (w.p.1) \] 
\end{cor} 

\begin{proof} 
The idea is to 
define the process $\tilde{X}(t) \defequiv X(t) - \expect{X(t)|\script{H}_X(t)}$, and then apply the result of 
Theorem \ref{lem:martingale} to the process $\tilde{X}(t)$.  This is shown in Appendix A for completeness. 
\end{proof}

\subsection{A Generalized Drift-Plus-Penalty Theorem} 

Now let $\Psi(t)$ be a non-negative stochastic process defined over slots $t \in \{0, 1, 2, \ldots\}$, and let 
$\beta(t)$ be another stochastic process defined on the same probability space and whose time average we want
to show is non-negative.   The $\Psi(t)$ process can represent the values of a general 
Lyapunov function over time $t \in \{0, 1, 2, \ldots\}$.  Define $\delta(t) \defequiv \Psi(t+1) - \Psi(t)$
as the difference process.  Define the history $\script{H}(t)$ for this system by: 
\begin{equation} \label{eq:history-general-dpp} 
\script{H}(t) \defequiv \{\Psi(0), \ldots, \Psi(t), \beta(0), \ldots, \beta(t-1)\}
\end{equation} 

\begin{thm} \label{thm:dpp-general} (Generalized Drift-Plus-Penalty)  Suppose $\Psi(0)$ is finite with probability 1,
that $\expect{\delta(t)^2}$ and $\expect{\beta(t)^2}$ are 
finite for all $t$, and that: 
\[ \sum_{t=1}^{\infty} \frac{\expect{\delta(t)^2}+\expect{\beta(t)^2}}{t^2} < \infty    \]
Further suppose that the following drift-plus-penalty condition holds for all $t$ and all possible $\script{H}(t)$: 
\begin{equation} \label{eq:dpp-gen} 
\expect{\delta(t) + \beta(t)| \script{H}(t)} \leq 0 
\end{equation} 
Then:
 \[ \limsup_{t\rightarrow\infty} \frac{1}{t}\sum_{\tau=0}^{t-1} \beta(\tau) \leq 0 \: \: (w.p.1) \]
 \end{thm} 
 
\begin{proof} 
Define $X(t) \defequiv \delta(t) + \beta(t)$.  The idea is to apply Corollary \ref{lem:supermartingale} to the process
$X(t)$.   To this end, we simply need to show that  $X(t)$ satisfies the assumptions needed in Corollary 
\ref{lem:supermartingale}. 
Note that the history $\script{H}(t)$ contains more information that the 
history $\script{H}_X(t)$, defined: 
\[ \script{H}_X(t) \defequiv \{X(0), X(1), \ldots, X(t-1)\} \]
Indeed,  $\script{H}_X(t)$ can be ascertained with knowledge of the more detailed history $\script{H}(t)$. Thus, 
we can write $\script{H}(t) = \script{H}(t) \cup \script{H}_X(t)$, as adding the information $\script{H}_X(t)$ does not
create any new information. Thus, using iterated expectations yields: 
\begin{eqnarray*}
\expect{\expect{X(t)|\script{H}(t)}|\script{H}_X(t)} &=& \expect{\expect{X(t)|\script{H}(t)\cup\script{H}_X(t)}|\script{H}_X(t)} \\
&=& \expect{X(t) | \script{H}_X(t)} 
\end{eqnarray*}
On the other hand, by (\ref{eq:dpp-gen}) we have: 
\begin{eqnarray*}
 \expect{\expect{X(t)|\script{H}(t)}|\script{H}_X(t)} &=& \expect{\expect{\delta(t) + \beta(t)|\script{H}(t)}|\script{H}_X(t)} \\
 &\leq& \expect{ 0 | \script{H}_X(t)} = 0 
 \end{eqnarray*}
Therefore, for all $t$ and all possible $\script{H}_X(t)$ we have: 
\[ \expect{X(t)|\script{H}_X(t)} \leq 0 \]
It remains only to show that: 
\[ \sum_{t=1}^{\infty}\frac{\expect{X(t)^2}}{t^2} < \infty \]
Because $(\delta(t) + \beta(t))^2 \leq 2\delta(t)^2 + 2\beta(t)^2$, we have:
\begin{eqnarray*}
\sum_{t=1}^{\infty}\frac{\expect{X(t)^2}}{t^2} &=& \sum_{t=1}^{\infty} \frac{\expect{(\delta(t) + \beta(t))^2}}{t^2} \\
&\leq& 2\sum_{t=1}^{\infty}\frac{\expect{\delta(t)^2 + \beta(t)^2}}{t^2} < \infty
\end{eqnarray*}
Thus, by  Corollary \ref{lem:supermartingale} we have: 
\begin{equation} \label{eq:use-this-dpp-gen} 
 \limsup_{t\rightarrow\infty} \frac{1}{t} \sum_{\tau=0}^{t-1}  X(\tau) \leq 0 \: \: (w.p.1) 
 \end{equation} 
However, recalling that $X(t) \defequiv \Psi(t+1) - \Psi(t) + \beta(t)$, we have: 
\begin{eqnarray*}
 \sum_{\tau=0}^{t-1} X(\tau) &=& \Psi(t) - \Psi(0)  +  \sum_{\tau=0}^{t-1} \beta(\tau) \\
 &\geq&  -\Psi(0) + \sum_{\tau=0}^{t-1} \beta(\tau) 
 \end{eqnarray*}
where the final inequality holds because $\Psi(t) \geq 0$. Dividing the above inequality by $t$ yields: 
\[\frac{1}{t}\sum_{\tau=0}^{t-1} X(\tau) \geq  \frac{-\Psi(0)}{t}  +  \frac{1}{t}\sum_{\tau=0}^{t-1} \beta(\tau) \]
Taking a $\limsup$ of the above as $t \rightarrow\infty$ and using (\ref{eq:use-this-dpp-gen}) proves the result. 
\end{proof}

\section{The Lyapunov Optimization Theorem --- Proving Theorem \ref{thm:lyap-opt}} \label{section:main} 

Consider the stochastic processes $\bv{Q}(t) = (Q_1(t), \ldots, Q_K(t))$ and $p(t)$ as described in Section 
\ref{section:background}.  Consider the quadratic Lyapunov function $L(\bv{Q}(t))$ defined in (\ref{eq:lyapunov-function}), repeated
again here for convenience: 
\[ L(\bv{Q}(t)) = \frac{1}{2}\sum_{k=1}^Kw_k Q_k(t)^2 \]
where $w_k>0$ for all $k$. 
Define $\norm{\bv{Q}(t)}$ by: 
\begin{eqnarray*}
& \norm{\bv{Q}(t)} \defequiv \sqrt{L(\bv{Q}(t))} = \sqrt{\frac{1}{2}\sum_{k=1}^Kw_kQ_k(t)^2} & 
\end{eqnarray*} 
It is not difficult to show that: 
\begin{eqnarray}
& \sum_{k=1}^K \frac{\sqrt{w_k}}{\sqrt{2}}|Q_k(t)| \geq \norm{\bv{Q}(t)} & \label{eq:triangle} 
 \end{eqnarray}
Further, for any vectors $\bv{a}, \bv{b}$ we have: 
\begin{equation} \label{eq:t2} 
\norm{\bv{a} + \bv{b}} \leq \norm{\bv{a}} + \norm{\bv{b}}
\end{equation} 

Define the drift $\Delta(\script{H}(t))$ according to (\ref{eq:lyapunov-drift}), where the history $\script{H}(t)$ is defined 
in (\ref{eq:history-Q}). 
 Define the Lyapunov difference process 
$\delta(t) \defequiv L(\bv{Q}(t+1)) - L(\bv{Q}(t))$, and note by definition that: 
\begin{equation} \label{eq:deltaD} 
 \expect{\delta(t)|\script{H}(t)} = \Delta(\script{H}(t)) 
 \end{equation} 
Define $d_k(t)$ as the queue $k$ difference process: 
\[ d_k(t) \defequiv Q_k(t+1) - Q_k(t) \]
We will bound the time averages of $p(t)$ and $Q_k(t)$ when the following drift-plus-penalty condition holds
for all $t$ and all $\script{H}(t)$: 
\begin{equation} \label{eq:vldpp}
 \Delta(\script{H}(t)) + V\expect{p(t)|\script{H}(t)} \leq B + Vp^* - \epsilon \sum_{k=1}^K |Q_k(t)| 
 \end{equation} 
for some finite constants $B$, $p^*$, $V$, $\epsilon$.  To this end, we define $\Psi(t) \defequiv L(\bv{Q}(t))$
and $\beta(t)$ as follows: 
\begin{equation} \label{eq:beta-lyap}
 \beta(t) \defequiv Vp(t) - B - Vp^* + \epsilon\sum_{k=1}^K|Q_k(t)|
 \end{equation} 
 The idea is to show that the assumptions needed in the generalized drift-plus-penalty theorem (Theorem \ref{thm:dpp-general}) 
 hold for these definitions
 of $\Psi(t)$ and $\beta(t)$.

\begin{thm} \label{thm:var-lyap}  
Suppose that  the boundedness assumptions (\ref{eq:boundedness1}) and (\ref{eq:boundedness2}) hold. 
Suppose that $\expect{Q_k(t)^2}$ is finite for all $k$ and all $t$, and that for all $k \in \{1, \ldots, K\}$ we have: 
\begin{equation} \label{eq:var-lyap}
 \sum_{t=1}^{\infty} \frac{\expect{Q_k(t)^2}}{t^2} < \infty  
 \end{equation} 
Define the quadratic Lyapunov function 
 $L(\bv{Q}(t))$ as in (\ref{eq:lyapunov-function}), and suppose there are constants $B$, $p^*$, $V\geq0$, $\epsilon\geq0$ for which the
drift-plus-penalty condition (\ref{eq:vldpp}) holds for all $t$ and all possible $\script{H}(t)$.  Then:

a) If $V>0$ we have: 
\begin{eqnarray}
\limsup_{t\rightarrow\infty} \frac{1}{t}\sum_{\tau=0}^{t-1} p(\tau) &\leq& p^* + \frac{B}{V} \: \: (w.p.1)  \label{eq:expect13} 
\end{eqnarray}

b) If $\epsilon>0$, we have: 
\begin{eqnarray} 
\limsup_{t\rightarrow\infty} \frac{1}{t}\sum_{\tau=0}^{t-1} \sum_{k=1}^K|Q_k(\tau)| &\leq& \frac{B}{\epsilon} + \frac{V}{\epsilon}\limsup_{t\rightarrow\infty}\frac{1}{t}\sum_{\tau=0}^{t-1}[p^* - p(\tau)]  \: \: (w.p.1) \label{eq:expect23} 
\end{eqnarray} 
\end{thm} 

\begin{proof}
Define $\Psi(t) \defequiv L(\bv{Q}(t))$, 
$\delta(t) \defequiv L(\bv{Q}(t+1)) - L(\bv{Q}(t))$ and define $\beta(t)$ as in (\ref{eq:beta-lyap}).
For all $t$ and all $\script{H}(t)$ we have: 
\begin{eqnarray}
 \expect{\delta(t) + \beta(t)|\script{H}(t)}  &=& \Delta(\script{H}(t)) + \expect{\beta(t)|\script{H}(t)} \label{eq:vl1} \\
 &=& \Delta(\script{H}(t)) + V\expect{p(t)|\script{H}(t)} - B - Vp^* + \epsilon\sum_{k=1}^K|Q_k(t)|\label{eq:vl2} \\
 &\leq& 0 \label{eq:vl3} 
 \end{eqnarray}
 where (\ref{eq:vl1}) follows from (\ref{eq:deltaD}), (\ref{eq:vl2}) follows by definition of $\beta(t)$ and the
 fact that $\expect{|Q_k(t)||\script{H}(t)} = |Q_k(t)|$,  and (\ref{eq:vl3}) follows
 from (\ref{eq:vldpp}). 

\emph{Claim 1:} 
\[ \sum_{t=1}^{\infty}\frac{\expect{\delta(t)^2 + \beta(t)^2}}{t^2} < \infty \]
This claim is proven in Appendix B.  Assuming the result of the claim, we know that all conditions for the $\Psi(t)$ and
$\beta(t)$ processes needed to apply Theorem \ref{thm:dpp-general} hold.  We thus conclude: 
\begin{equation*} 
\limsup_{t\rightarrow\infty} \frac{1}{t} \sum_{\tau=0}^{t-1} \beta(\tau) \leq 0 \: \: (w.p.1) 
\end{equation*} 
That is: 
\begin{equation} \label{eq:lots} 
\limsup_{t\rightarrow\infty} \frac{1}{t} \sum_{\tau=0}^{t-1}  \left[ Vp(\tau) - B - Vp^* + \epsilon\sum_{k=1}^K|Q_k(\tau)| \right] \leq 0  \: \: (w.p.1) 
\end{equation} 
First assume that $V>0$. 
Neglecting the non-negative term $\epsilon \sum_{k=1}^K|Q_k(\tau)|$ from (\ref{eq:lots}) and dividing by $V$
yields: 
\[  \limsup_{t\rightarrow\infty} \frac{1}{t}  \sum_{\tau=0}^{t-1}[ p(\tau) - B/V - p^*]  \leq 0 \: \: (w.p.1)   \]
This proves (\ref{eq:expect13}). 

Now note that for any functions $f(t)$, $g(t)$, we have:\footnote{This follows by: $\limsup_{t\rightarrow\infty} f(t) = \limsup_{t\rightarrow\infty} [g(t) + (f(t) - g(t))] \leq \limsup_{t\rightarrow\infty} g(t) + \limsup_{t\rightarrow\infty} (f(t) - g(t))$.} 
\[ \limsup_{t\rightarrow\infty} [f(t) - g(t)] \leq 0  \implies \limsup_{t\rightarrow\infty} f(t) \leq \limsup_{t\rightarrow\infty} g(t) \]
Defining $f(t) \defequiv \frac{\epsilon}{t}\sum_{\tau=0}^{t-1}\sum_{k=1}^K|Q_k(\tau)|$ and $g(t) \defequiv \frac{1}{t}\sum_{\tau=0}^{t-1}[B + V(p^* - p(\tau))]$, it follows from (\ref{eq:lots}) that:
\[ \limsup_{t\rightarrow\infty} \frac{\epsilon}{t} \sum_{\tau=0}^{t-1} \sum_{k=1}^K|Q_k(\tau)|\leq 
\limsup_{t\rightarrow\infty} \frac{1}{t}\sum_{\tau=0}^{t-1}[B + V(p^* - p(\tau))]   \: \: (w.p.1) \]
If $\epsilon>0$, we can divide the above by $\epsilon$ to prove (\ref{eq:expect23}).
\end{proof}

\begin{thm} \label{thm:quad-lyap} 
Suppose we have a quadratic Lyapunov function $L(\bv{Q}(t))$ as defined in (\ref{eq:lyapunov-function}), and 
that assumption (\ref{eq:boundedness2}) holds, so that $\expect{d_k(t)^4|\bv{Q}(t)} \leq D$ for all $t$ and for
some finite constant $D$, 
 where $d_k(t) = Q_k(t+1) - Q_k(t)$.  Suppose that $\expect{\norm{\bv{Q}(0)}^4} < \infty$.
 Suppose that there is an $\epsilon>0$ and a constant $\tilde{B}>0$ such that: 
 \begin{equation} \label{eq:quad-lyap-drift} 
  \Delta(\script{H}(t)) \leq \tilde{B} - \epsilon\sum_{k=1}^K|Q_k(t)| 
  \end{equation} 
 Then:

a) There are constants $c>0$ and $a>0$ such that whenever $\norm{\bv{Q}(t)} \geq a$, we have: 
\[ \expect{\norm{\bv{Q}(t+1)}|\bv{Q}(t)} \leq \norm{\bv{Q}(t)} - c  \]

b) There is a finite constant $b>0$ 
such that for all $M \in \{1, 2, \ldots\}$ we have: 
 \[ \frac{1}{M} \sum_{t=0}^{M-1} \expect{\norm{\bv{Q}(t)}^3} \leq b \]

c) For all $k \in \{1, \ldots, K\}$ we have: 
\begin{equation*} 
 \sum_{t=1}^{\infty} \frac{\expect{Q_k(t)^2}}{t^2} < \infty 
 \end{equation*} 

d) For all $k \in \{1, \ldots, K\}$ we have: 
\[ \lim_{t\rightarrow\infty} \frac{Q_k(t)}{t}  = 0 \: \: (w.p.1) \]

e) We have: 
\[ \limsup_{t\rightarrow\infty} \frac{1}{t}\sum_{\tau=0}^{t-1} \sum_{k=1}^K|Q_k(\tau)| \leq \frac{\tilde{B}}{\epsilon} \: \: (w.p.1) \]
\end{thm} 

\begin{proof} 
The proof of parts (a) and (b) closely follow a similar result derived for exponential Lyapunov functions with 
deterministically bounded queue changes in \cite{longbo-lagrange}, and are provided in Appendix C.  
To prove parts (c), (d), (e), we have from part (b) 
that for all $M \in \{1, 2, 3, \ldots\}$: 
 \begin{equation} \label{eq:reuse0} 
  \sum_{t=0}^{M-1} \expect{\norm{\bv{Q}(t)}^3} \leq bM 
  \end{equation} 
  However, we have $\norm{\bv{Q}(t)}^3 \geq \norm{\bv{Q}(t)}^2 - 1$.  Using this with (\ref{eq:reuse0}) gives: 
\[ \sum_{t=0}^{M-1} (\expect{\norm{\bv{Q}(t)}^2} - 1) \leq bM \] 
  and so: 
  \[ \sum_{t=0}^{M-1} \expect{\norm{\bv{Q}(t)}^2} \leq (b+1)M \]
  Using $\frac{w_k}{2}Q_k(t)^2 \leq \norm{\bv{Q}(t)}^2$ in the above inequality proves that there is a finite
  constant $C>0$ such that for all $k \in \{1, \ldots, K\}$ we have: 
  \begin{equation*} 
\sum_{t=0}^{M-1}  \expect{Q_k(t)^2}  \leq CM \: \: \forall M \in \{1, 2, 3, \ldots\} 
\end{equation*} 
Lemma \ref{lem:appd} in Appendix D shows that the above inequality implies the result of part (c). 
  
Part (d) follows immediately from the result of part (c) together with  Corollary \ref{cor:rate-stable}. 
 To prove part (e), we note that the result of part (c) implies that the conditions for Theorem \ref{thm:var-lyap} are met
 for the case $\epsilon>0$, $p(t) = p^* = V = 0$, $B = \tilde{B}$, which yields the result. 
\end{proof}

\subsection{Completing the proof of Theorem \ref{thm:lyap-opt}} 

Suppose now the assumptions of Theorem \ref{thm:lyap-opt} hold, so that the drift-plus-penalty condition
(\ref{eq:dpp-condition}) is satisfied for all $t$ and all $\script{H}(t)$, and the boundedness assumptions (\ref{eq:boundedness1})-(\ref{eq:boundedness2}) hold.  We temporarily also assume that the initial state $\bv{Q}(0)$ is deterministically given as some
constant vector (so that $\expect{\norm{\bv{Q}(0)}^4} = \norm{\bv{Q}(0)}^4 < \infty$).  The condition (\ref{eq:dpp-condition}) 
together with the fact that $\expect{p(t)|\script{H}(t)} \geq p_{min}$  implies: 
\[ \Delta(\script{H}(t)) \leq B + V(p^* - p_{min}) - \epsilon \sum_{k=1}^K|Q_k(t)| \]
Defining $\tilde{B} = B + V(p^* - p_{min})$, by 
Theorem \ref{thm:quad-lyap} we know all queues are rate stable, that is, $\lim_{t\rightarrow\infty} Q_k(t)/t = 0$ with probability 1.  
We also know by Theorem \ref{thm:quad-lyap} that: 
\[ \sum_{t=1}^{\infty} \frac{\expect{Q_k(t)^2}}{t^2} < \infty \]
Then all assumptions are satisfied to apply Theorem \ref{thm:var-lyap}, and so we have that: 
\begin{eqnarray}
\limsup_{t\rightarrow\infty} \frac{1}{t}\sum_{\tau=0}^{t-1} p(\tau) &\leq& p^* + \frac{B}{V} \: \: (w.p.1) \nonumber \\
\limsup_{t\rightarrow\infty} \frac{1}{t}\sum_{\tau=0}^{t-1} \sum_{k=1}^K|Q_k(\tau)| &\leq& \frac{B}{\epsilon} + \frac{V}{\epsilon}\limsup_{t\rightarrow\infty}\frac{1}{t}\sum_{\tau=0}^{t-1}[p^* - p(\tau)]  \: \: (w.p.1)  \label{eq:togethernow} 
\end{eqnarray}
Because $\expect{-p(t) | \script{H}(t)} \leq  -p_{min}$ for all $t$ and all $\script{H}(t)$, and $\sum_{t=1}^{\infty} \expect{p(t)^2}/t^2 < \infty$, 
we know by  Corollary \ref{lem:supermartingale} that: 
\[ \limsup_{t\rightarrow\infty} \frac{1}{t}\sum_{\tau=0}^{t-1} [p^*-p(\tau)] \leq p^* - p_{min} \]
This together with  (\ref{eq:togethernow}) proves (\ref{eq:expect22}). 
 Thus, all desired performance bounds hold with probability 1 
under the assumption that the initial queue vector is some finite value $\bv{Q}(0)$. Because these bounds do not
depend on $\bv{Q}(0)$, it follows that these same bounds hold (with probability 1) 
if $\bv{Q}(0)$ is chosen randomly, provided that $\bv{Q}(0)$ is  finite
with probability 1. 

\subsection{Variations on Theorem \ref{thm:lyap-opt}} 

Suppose there are processes $B(t)$, $p(t)$, $p^*(t)$, $\bv{Q}(t)$ and constants $V\geq0$, $\epsilon>0$ such that for all 
$t$ and all possible $\script{H}(t)$, we have: 
\begin{equation} \label{eq:extension} 
\Delta(\script{H}(t)) + V\expect{p(t)|\script{H}(t)} \leq \expect{B(t)|\script{H}(t)} + V\expect{p^*(t)|\script{H}(t)} - \epsilon\sum_{k=1}^K|Q_k(t)|
\end{equation} 
This is a variation on the drift-plus-penalty condition (\ref{eq:dpp-condition}) that uses a time-varying $p^*(t)$ and $B(t)$. 
Suppose that $\bv{Q}(0)$ is finite with probability 1, and that: 
\begin{itemize} 
\item Second moments of $p(t)$, $B(t)$, and $p^*(t)$ are finite for all $t$, and: 
\[ \sum_{t=1}^{\infty} \frac{\expect{[V(p(t) - p^*(t)) - B(t)]^2}}{t^2} < \infty \]
\item There is a constant $\beta_{min}$ such that for all $t$ and all $\script{H}(t)$: 
\[ \expect{V(p(t) - p^*(t)) - B(t)|\script{H}(t)} \geq \beta_{min} \]
\item There is a constant $D>0$ such that for all $k \in \{1, \ldots, K\}$, all $t$,  and all possible $\bv{Q}(t)$: 
\[ \expect{(Q_k(t+1) - Q_k(t))^4|\bv{Q}(t)} \leq D \]
\end{itemize}
Then we can define $\tilde{B} \defequiv 0$, $\tilde{V}=1$, $\beta(t) \defequiv V(p(t) - p^*(t)) - B(t)$, $\beta^*=0$ to find: 
\[ \Delta(\script{H}(t)) +\expect{\beta(t)|\script{H}(t)} \leq 0 - \epsilon\sum_{k=1}^K|Q_k(t)| \]
Then the conditions of Theorem \ref{thm:lyap-opt} hold for $\beta(t)$ and $\beta^*$, and so we conclude (using (\ref{eq:togethernow})): 
\begin{eqnarray*}
\limsup_{t\rightarrow\infty} \frac{1}{t}\sum_{\tau=0}^{t-1} \beta(\tau) &\leq& 0 \: \: (w.p.1)  \\
\limsup_{t\rightarrow\infty} \frac{1}{t}\sum_{\tau=0}^{t-1}\sum_{k=1}^K|Q_k(\tau)| &\leq& \frac{1}{\epsilon}\limsup_{t\rightarrow\infty}\frac{1}{t}\sum_{\tau=0}^{t-1} [-\beta(\tau)] \: \: (w.p.1) 
\end{eqnarray*}
Thus: 
\begin{eqnarray*}
\limsup_{t\rightarrow\infty} \frac{1}{t}\sum_{\tau=0}^{t-1} [p(\tau) - p^*(\tau)] &\leq& \frac{1}{V}\limsup_{t\rightarrow\infty} \frac{1}{t}\sum_{\tau=0}^{t-1}B(\tau) \: \: (w.p.1)  \\
\limsup_{t\rightarrow\infty} \frac{1}{t}\sum_{\tau=0}^{t-1}\sum_{k=1}^K|Q_k(\tau)| &\leq& \frac{1}{\epsilon}\limsup_{t\rightarrow\infty}\frac{1}{t}\sum_{\tau=0}^{t-1} [B(\tau) + V(p^*(\tau) - p(\tau))] \: \: (w.p.1) 
\end{eqnarray*}

\section{Applications} \label{section:application} 

Here we illustrate an important application of Theorem \ref{thm:lyap-opt} to optimization of time averages in 
stochastic queueing networks.  This is the same scenario treated in \cite{now}.  However, while the work in \cite{now}
obtains bounds on the time average expectations via Theorem \ref{thm:lyap-opt-expectation}, here we obtain bounds
on the pure time averages via Theorem \ref{thm:lyap-opt}. 

Consider a $K$ queue network with queue vector $\bv{Q}(t) = (Q_1(t), \ldots, Q_K(t))$ that evolves in slotted
time $t \in \{0, 1, 2, \ldots\}$ with update equation: 
\begin{equation} \label{eq:q-update} 
 Q_k(t+1) = \max[Q_k(t) - b_k(t) + a_k(t), 0] \: \: \forall k \in \{1, \ldots, K\} 
 \end{equation} 
where $a_k(t)$ and $b_k(t)$ are arrival and service variables, respectively, for queue $k$.  These are determined
on slot $t$ by general functions $\hat{a}_k(\alpha(t), \omega(t))$, $\hat{b}_k(\alpha(t), \omega(t))$ of a 
\emph{network state} $\omega(t)$ and a \emph{control action} $\alpha(t)$: 
\[ a_k(t) = \hat{a}_k(\alpha(t), \omega(t)) \: \: , \: \: b_k(t) = \hat{b}_k(\alpha(t), \omega(t)) \]
where the control action $\alpha(t)$ is make every slot $t$ with knowledge of the current $\omega(t)$ and is 
chosen within some abstract set $\script{A}_{\omega(t)}$.  The $\omega(t)$ value can represent random arrival and
channel state information on slot $t$, and $\alpha(t)$ can represent a resource allocation decision.  For simplicity, assume
the $\omega(t)$ process is i.i.d. over slots. 

The control action additionally incurs a vector of \emph{penalties} $\bv{y}(t) = (y_0(t), y_1(t), \ldots, y_M(t))$, again given by
general functions of $\alpha(t)$ and $\omega(t)$: 
\[ y_m(t) = \hat{y}_m(\alpha(t), \omega(t)) \]
For $t>0$, define $\overline{a}_k(t)$, $\overline{b}_k(t)$, $\overline{y}_m(t)$, $\overline{Q}_k(t)$ 
as time averages over the first $t$ slots: 
\begin{eqnarray*}
\overline{a}_k(t) \defequiv \frac{1}{t}\sum_{\tau=0}^{t-1} a_k(\tau) , & 
\overline{b}_k(t) \defequiv \frac{1}{t}\sum_{\tau=0}^{t-1} b_k(\tau) \\
\overline{y}_m(t) \defequiv \frac{1}{t}\sum_{\tau=0}^{t-1} y_m(\tau) , &
\overline{Q}_k(t) \defequiv \frac{1}{t}\sum_{\tau=0}^{t-1} Q_k(\tau) 
\end{eqnarray*}
The goal is to choose control actions $\alpha(t) \in \script{A}_{\omega(t)}$ over time to solve the following stochastic
network optimization problem: 
\begin{eqnarray}
\mbox{Minimize:} && \limsup_{t\rightarrow\infty}  \overline{y}_0(t)  \label{eq:opt1} \\
\mbox{Subject to:} &1)& \limsup_{t\rightarrow\infty} \overline{Q}_k(t) < \infty \: \: \forall k \in \{1, \ldots, K\} \label{eq:opt2} \\
&2)& \limsup_{t\rightarrow\infty} \overline{y}_m(t) \leq 0 \: \: \forall m \in \{1, \ldots, M\} \label{eq:opt3} \\
&3)& \alpha(t) \in \script{A}_{\omega(t)} \: \: \forall t \in \{0, 1, 2, \ldots\} \label{eq:opt4} 
\end{eqnarray}
Typical penalties can represent \emph{power expenditures}.  For example, suppose $y_m(t) \defequiv p_m(t) - p_m^{av}$, 
where $p_m(t)$ is the power incurred in component $m$ of the network on slot $t$, and $p_m^{av}$ is a required
time average 
power expenditure.   Then ensuring $\limsup_{t\rightarrow\infty} \overline{y}_m(t) \leq 0$ ensures that 
$\limsup_{t\rightarrow\infty} \overline{p}_m(t) \leq p_m^{av}$, so that the desired time average power constraint is 
met  \cite{neely-energy-it}. 

To ensure the time average penalty constraints are met, for each $m \in \{1, \ldots, M\}$ we define a \emph{virtual queue} 
$Z_m(t)$ as follows: 
\begin{equation} \label{eq:z-update} 
Z_m(t+1) = \max[Z_m(t) + y_m(t), 0] 
\end{equation} 
It is easy to see that for any $t>0$ we have: 
\[ Z_m(t) -Z_m(0) \geq \sum_{\tau=0}^{t-1} y_m(\tau) \]
and therefore, dividing by $t$ and rearranging terms yields: 
\[ \frac{1}{t}\sum_{\tau=0}^{t-1} y_m(\tau) \leq  \frac{Z_m(t)}{t} - \frac{Z_m(0)}{t} \]
It follows that if $Z_m(t)$ is rate stable for all $m$, so that $Z_m(t)/t \rightarrow 0$ with probability 1, then 
the constraint (\ref{eq:opt3}) is satisfied with probability 1. 

Now define $\bv{\Theta}(t) \defequiv [\bv{Q}(t), \bv{Z}(t)]$ as the combined queue vector, and 
define the Lyapunov function: 
\[ L(\bv{\Theta}(t)) \defequiv \frac{1}{2}\left[\sum_{k=1}^KQ_k(t)^2 + \sum_{m=1}^MZ_m(t)^2\right] \]
The system history $\script{H}(t)$ is defined: 
\[ \script{H}(t) \defequiv \{\bv{\Theta}(0), \bv{\Theta}(1), \ldots, \bv{\Theta}(t), y_0(0), y_0(1), \ldots, y_0(t-1)\} \]
The drift-plus-penalty algorithm thus seeks to minimize a bound on: 
\[ \Delta(\script{H}(t)) + V\expect{\hat{y}_0(\alpha(t), \omega(t))|\script{H}(t)} \]

\subsection{Computing the Drift-Plus-Penalty Inequality} 

Assume the functions $\hat{a}_k(\cdot)$, $\hat{b}_k(\cdot)$, $\hat{y}_0(\cdot)$ satisfy the following 
for all possible $\omega(t)$ and all possible $\alpha(t) \in \script{A}_{\omega(t)}$: 
\begin{eqnarray*}
 0 \leq \hat{a}_k(\alpha(t), \omega(t)) \: \: , \: \: 0 \leq \hat{b}_k(\alpha(t), \omega(t)) \: \: , \: \: \hat{y}_0(\alpha(t), \omega(t)) \geq y_0^{min} 
 \end{eqnarray*}
 where $y_0^{min}$ is a deterministic lower bound on $y_0(t)$ for all $t$. 
Also assume that there is a finite constant $D>0$ such that 
for all (possibly randomized) choices of $\alpha(t)$ in reaction to the i.i.d. $\omega(t)$ we have: 
\begin{eqnarray}
\expect{\hat{a}_k(\alpha(t), \omega(t))^4} &\leq& D \: \: \forall k \in \{1, \ldots, K\} \label{eq:mb1} \\
\expect{\hat{b}_k(\alpha(t), \omega(t))^4} &\leq& D \: \: \forall k \in \{1, \ldots, K\} \label{eq:mb2} \\
\expect{\hat{y}_m(\alpha(t), \omega(t))^4} &\leq& D \: \: \forall m \in \{1, \ldots, M\} \label{eq:mb3} \\
\expect{\hat{y}_0(\alpha(t), \omega(t))^2} &\leq& D \label{eq:mb4} 
\end{eqnarray}
where the expectations are taken with respect to the distribution of the i.i.d. $\omega(t)$ process, and the possibly
randomized decisions $\alpha(t) \in \script{A}_{\omega(t)}$. 

By squaring (\ref{eq:q-update}) and (\ref{eq:z-update}) it is not difficult to show that the drift-plus-penalty expression satisfies 
the following bound (see \cite{now}): 
\begin{eqnarray}
\Delta(\script{H}(t)) + V\expect{\hat{y}_0(\alpha(t), \omega(t))|\script{H}(t)} &\leq& B + V\expect{\hat{y}_0(\alpha(t), \omega(t))|\script{H}(t)}  \nonumber \\
&&+ \sum_{k=1}^KQ_k(t)\expect{\hat{a}_k(\alpha(t), \omega(t)) - \hat{b}_k(\alpha(t), \omega(t))|\script{H}(t)}  \nonumber \\
&& + \sum_{m=1}^MZ_m(t)\expect{\hat{y}_m(\alpha(t), \omega(t))|\script{H}(t)}  \label{eq:dpp-opt1} 
\end{eqnarray}
for some finite constant $B>0$, representing a sum on the second moment bounds of the $a_k(t)$, $b_k(t)$, and $y_m(t)$ processes. 

\subsection{The Dynamic Drift-Plus-Penalty Algorithm} 

It is easy to show that the right-hand-side of the  inequality (\ref{eq:dpp-opt1}) 
is minimized by the policy that, every slot $t$, 
observes only the current 
queue values $\bv{Q}(t)$, $\bv{Z}(t)$ and the current $\omega(t)$ and chooses $\alpha(t) \in \script{A}_{\omega(t)}$ to minimize
the following expression: 
\begin{eqnarray*}
V\hat{y}_0(\alpha(t), \omega(t)) + \sum_{k=1}^KQ_k(t)[\hat{a}_k(\alpha(t), \omega(t)) - \hat{b}_k(\alpha(t), \omega(t))] 
 + \sum_{m=1}^MZ_m(t)\hat{y}_m(\alpha(t), \omega(t))
\end{eqnarray*}
Then update the actual queues $Q_k(t)$ 
according to (\ref{eq:q-update}) and the virtual queues $Z_m(t)$ according
to (\ref{eq:z-update}). 
This policy does not require knowledge of the probability distribution for $\omega(t)$.  One difficulty is that
it may not be possible to achieve the infimum of the above expression over the set $\script{A}_{\omega(t)}$, because
we are using general (possibly non-continuous) functions $\hat{a}_k(\alpha(t), \omega(t))$, $\hat{b}_k(\alpha(t), \omega(t))$, 
$\hat{y}_m(\alpha(t), \omega(t))$ and a general (possibly non-compact) set $\script{A}_{\omega(t)}$.   Thus, we simply
assume there is a finite constant $C\geq0$ such that our algorithm chooses $\alpha(t) \in \script{A}_{\omega(t)}$ to come within an additive constant $C$ of the infimum on every slot $t$, so that: 
\begin{eqnarray*}
V\hat{y}_0(\alpha(t), \omega(t)) + \sum_{k=1}^KQ_k(t)[\hat{a}_k(\alpha(t), \omega(t)) - \hat{b}_k(\alpha(t), \omega(t))] 
 + \sum_{m=1}^MZ_m(t)\hat{y}_m(\alpha(t), \omega(t)) \\
 \leq C + \inf_{\alpha\in\script{A}_{\omega(t)}} \left[V\hat{y}_0(\alpha, \omega(t)) + \sum_{k=1}^KQ_k(t)[\hat{a}_k(\alpha, \omega(t)) - \hat{b}_k(\alpha, \omega(t))] 
 + \sum_{m=1}^MZ_m(t)\hat{y}_m(\alpha, \omega(t))\right] 
 \end{eqnarray*}
Such a choice of $\alpha(t)$ is called a \emph{$C$-additive approximation.}   The case $C=0$ corresponds to achieving
the exact infimum every slot. 

\subsection{$\omega$-only policies} 

Define a $\omega$-only policy to be one that chooses $\alpha(t) \in \script{A}_{\omega(t)}$ every slot $t$ according to a stationary and randomized
decision based only on the observed $\omega(t)$ (in particular, being independent of $\script{H}(t)$). 
Assume there exists an $\epsilon>0$ and a particular $\omega$-only policy $\alpha^*(t)$ that yields the following: 
\begin{eqnarray}
\expect{\hat{a}_k(\alpha^*(t), \omega(t)) - \hat{b}_k(\alpha^*(t), \omega(t))} &\leq& -\epsilon  \: \: \forall k \in \{1, \ldots, K\}  \label{eq:omega-only1first}\\
\expect{\hat{y}_m(\alpha^*(t), \omega(t))} &\leq& - \epsilon \: \: \forall m \in \{1, \ldots, M\} \label{eq:omega-only2first} 
\end{eqnarray}
Under this assumption, it can be shown that the algorithm that uses the $\omega$-only decisions $\alpha^*(t)$ every slot $t$
satisfies the constraints (\ref{eq:opt2})-(\ref{eq:opt4}), and hence the problem (\ref{eq:opt1})-(\ref{eq:opt3}) is \emph{feasible} (meaning
that its constraints are possible to satisfy). 
Further, this assumption (similar to a Slater assumption in convex optimization theory \cite{bertsekas-nonlinear}) is only 
slightly stronger than what is required for feasibility.  Indeed, it can be shown that if the problem (\ref{eq:opt1})-(\ref{eq:opt3}) 
is feasible, then for all $\delta>0$ there must be an $\omega$-only algorithm that satisfies  \cite{sno-text}:
\begin{eqnarray*}
\expect{\hat{a}_k(\alpha^*(t), \omega(t)) - \hat{b}_k(\alpha^*(t), \omega(t))} &\leq& \delta  \: \: \forall k \in \{1, \ldots, K\}  \\
\expect{\hat{y}_m(\alpha^*(t), \omega(t))} &\leq& \delta \: \: \forall m \in \{1, \ldots, M\} 
\end{eqnarray*}

Define $\epsilon_{max}$ as the supremum of all $\epsilon$ values for which an $\omega$-only policy exists and 
satisfies (\ref{eq:omega-only1first})-(\ref{eq:omega-only2first}).   For $0 \leq \epsilon \leq \epsilon_{max}$, define $y_0^{opt}(\epsilon)$
as the infimum value of $y$ such that for all $\delta>0$, there exists an $\omega$-only policy $\alpha^*(t)$ that satisfies the 
following constraints: 
\begin{eqnarray*} 
\expect{\hat{y}_0(\alpha^*(t), \omega(t))} &\leq& y  + \delta \\
\expect{\hat{a}_k(\alpha^*(t), \omega(t)) - \hat{b}_k(\alpha^*(t), \omega(t))} &\leq& -\epsilon + \delta  \: \: \forall k \in \{1, \ldots, K\}  \\
\expect{\hat{y}_m(\alpha^*(t), \omega(t))} &\leq& - \epsilon + \delta \: \: \forall m \in \{1, \ldots, M\}
\end{eqnarray*} 
It is not difficult to show that:
\begin{itemize}  
\item These constraints are feasible whenever $0 \leq \epsilon \leq \epsilon_{max}$. 
\item  The function 
$y_0(\epsilon)$ is finite, continuous, and non-decreasing on the interval 
$0 \leq \epsilon \leq \epsilon_{max}$.  
\item The set of all such $y$ values that satisfy the above constraints is closed.  
\end{itemize} 
Thus, whenever $0 \leq \epsilon \leq \epsilon_{max}$, 
for any $\delta>0$ there exists an $\omega$-only algorithm $\alpha^*(t)$ such that: 
\begin{eqnarray} 
\expect{\hat{y}_0(\alpha^*(t), \omega(t))} &\leq& y_0^{opt}(\epsilon)  + \delta \label{eq:omega-only0} \\
\expect{\hat{a}_k(\alpha^*(t), \omega(t)) - \hat{b}_k(\alpha^*(t), \omega(t))} &\leq& -\epsilon + \delta  \: \: \forall k \in \{1, \ldots, K\} \label{eq:omega-only1} \\
\expect{\hat{y}_m(\alpha^*(t), \omega(t))} &\leq& - \epsilon + \delta \: \: \forall m \in \{1, \ldots, M\} \label{eq:omega-only2} 
\end{eqnarray} 
It can be shown that $y_0^{opt}(0)$ is the infimum time average penalty for $y_0(t)$ over \emph{all} 
algorithms that meet the constraints 
(\ref{eq:opt2})-(\ref{eq:opt4}) (not just $\omega$-only algorithms) \cite{neely-energy-it}\cite{sno-text}.
Thus, we define $y_0^{opt} \defequiv y_0^{opt}(0)$.

\subsection{Performance Bounds} 

Because our policy $\alpha(t)$ comes within $C\geq 0$ of minimizing 
the right-hand-side of (\ref{eq:dpp-opt1}) every slot $t$ (given the observed
$\script{H}(t)$), we have for all $t$ and all possible $\script{H}(t)$: 
\begin{eqnarray*}
\Delta(\script{H}(t)) + V\expect{\hat{y}_0(\alpha(t), \omega(t))|\script{H}(t)} &\leq& B + C + V\expect{\hat{y}_0(\alpha^*(t), \omega(t))|\script{H}(t)} \nonumber \\
&& + \sum_{k=1}^KQ_k(t)\expect{\hat{a}_k(\alpha^*(t), \omega(t)) - \hat{b}_k(\alpha^*(t), \omega(t))|\script{H}(t)}  \nonumber \\
&& + \sum_{m=1}^MZ_m(t)\expect{\hat{y}_m(\alpha^*(t), \omega(t))|\script{H}(t)}  
\end{eqnarray*}
where $\alpha^*(t)$ is any other decision that can be implemented on slot $t$. Now fix $\epsilon$ in the interval 
$0 < \epsilon \leq  \epsilon_{max}$. Fix any $\delta>0$. Using the policy $\alpha^*(t)$ designed to 
achieve (\ref{eq:omega-only0})-(\ref{eq:omega-only2}) and noting
that this policy makes decisions independent of $\script{H}(t)$ yields: 
\begin{eqnarray*}
&&\hspace{-1in}\Delta(\script{H}(t)) + V\expect{\hat{y}_0(\alpha(t), \omega(t))|\script{H}(t)} \leq \nonumber \\
&&  B + C + V(y_0^{opt}(\epsilon) + \delta) 
-(\epsilon-\delta)\sum_{k=1}^KQ_k(t) -(\epsilon-\delta)\sum_{m=1}^MZ_m(t) 
\end{eqnarray*}
The above holds for all $\delta>0$.  Taking a limit as $\delta\rightarrow0$ yields:
\begin{eqnarray}
\Delta(\script{H}(t)) + V\expect{y_0(t)|\script{H}(t)} &\leq& B + C + Vy_0^{opt}(\epsilon) 
-\epsilon\sum_{k=1}^KQ_k(t) -\epsilon\sum_{m=1}^MZ_m(t) \label{eq:dpp-opt2} 
\end{eqnarray}
where for simplicity we have substituted  $y_0(t) = \hat{y}_0(\alpha(t), \omega(t))$ on the left-hand-side. 
Inequality (\ref{eq:dpp-opt2})
is in the exact form of the drift-plus-penalty condition (\ref{eq:dpp-condition}).  Recall that the penalty $y_0(t)$ is deterministically
lower bounded by some finite (possibly negative) value $y_0^{min}$.  Further, the moment
bounds (\ref{eq:mb1})-(\ref{eq:mb4}) can easily be shown to imply that the boundedness assumptions (\ref{eq:boundedness1})-(\ref{eq:boundedness2}) hold.  Thus, we can apply Theorem \ref{thm:lyap-opt} to conclude that all queues are rate stable (in particular
$Z_m(t)/t \rightarrow 0$ with probability 1 for all $k$, so that the constraints (\ref{eq:opt3}) are satisfied:
\[ \limsup_{t\rightarrow\infty} \overline{y}_m(t) \leq 0 \: \: \forall m \in \{1, \ldots, M\} \: \: (w.p.1) \]
  Further:  
\begin{eqnarray*} 
\limsup_{t\rightarrow\infty} \frac{1}{t}\sum_{\tau=0}^{t-1} y_0(\tau) &\leq& y_0^{opt}(\epsilon) + (B+C)/V  \: \: (w.p.1) \\
\limsup_{t\rightarrow\infty} \frac{1}{t}\sum_{\tau=0}^{t-1} \left[\sum_{k=1}^KQ_k(\tau) + \sum_{m=1}^MZ_m(\tau)\right] &\leq& \frac{B + C + V(y_0^{opt}(\epsilon) - y_{min})}{\epsilon} \: \: (w.p.1) 
\end{eqnarray*}
However, the above two bounds hold for all $\epsilon$ such that $0 < \epsilon \leq \epsilon_{max}$, and hence the two performance
bounds can be optimized separately over this interval.  Taking a limit as $\epsilon\rightarrow 0$  in the first bound 
and noting by continuity  that $\lim_{\epsilon\rightarrow0} y_0^{opt}(\epsilon) = y_0^{opt}(0) \defequiv y_0^{opt}$ yields: 
\begin{equation} \label{eq:power-bound} 
 \limsup_{t\rightarrow\infty} \frac{1}{t}\sum_{\tau=0}^{t-1} y_0(\tau) \leq  y_0^{opt}  + (B+C)/V \: \: (w.p.1) 
 \end{equation} 
Using $\epsilon=\epsilon_{max}$ in the second bound yields: 
\begin{equation} \label{eq:queue-bound} 
 \limsup_{t\rightarrow\infty} \frac{1}{t}\sum_{\tau=0}^{t-1} \left[\sum_{k=1}^KQ_k(\tau) + \sum_{m=1}^MZ_m(\tau)\right] \leq \frac{B +C+ V(y_0^{opt}(\epsilon_{max}) - y_{min})}{\epsilon_{max}} \: \: (w.p.1)  
 \end{equation} 

Thus, this simple dynamic algorithm satisfies the desired time average penalty 
constraints, stabilizes all queues $Q_k(t)$, and yields a time average penalty for $y_0(t)$ that is 
within $B/V$ of the optimal value $y_0^{opt}$.  The performance gap 
$B/V$ can be made arbitrarily small by choosing the $V$ parameter large (as shown
by (\ref{eq:power-bound})). 
The tradeoff is a time average queue backlog that is $O(V)$ (as shown by (\ref{eq:queue-bound})).  

By (\ref{eq:togethernow}), 
the bound (\ref{eq:queue-bound}) can be improved, at the expense of sometimes 
making it less easy to compute, by replacing
``$-y_{min}$'' on the right-hand-side with ``$-\liminf_{t\rightarrow\infty} \frac{1}{t}\sum_{\tau=0}^{t-1} y_0(\tau)$.''
Further, we note that the concept
of \emph{place-holder backlog} from \cite{neely-asilomar08} is compatible with this 
analysis and can often be used together with the above to provide 
improved backlog bounds.

\section{Conclusions} 

This work derives an extended 
drift-plus-penalty theorem for discrete time queueing systems. The theorem 
ensures all queues satisfy all major forms of stability, and that 
time averages meet desired constraints with probability 1. 
This extends  prior results that were known to hold only for time average expectations.  
The boundedness conditions required for the theorem are mild  and easily checked.  
In particular, the theorem applies to systems with an uncountably infinite
number of possible events, to Markov systems with an uncountably infinite state space (possibly neither irreducible nor aperiodic), and to non-Markov systems.  Our analysis combined the Kolmogorov law of large numbers for  martingale differences with
the drift-plus-penalty method from Lyapunov optimization.   
The results are applicable to a broad class of stochastic queueing networks, and are also 
useful in other contexts. 

%


\section*{Appendix A --- Proof of Corollary \ref{lem:supermartingale}} 

Suppose the assumptions of Corollary \ref{lem:supermartingale} hold, so that $\expect{X(t)|\script{H}_X(t)} \leq B$ for all
$t$ and all $\script{H}_X(t)$, and: 
\[ \sum_{t=1}^{\infty}\frac{\expect{X(t)^2}}{t^2} < \infty \] 
Define $\tilde{X}(t) \defequiv X(t) - \expect{X(t)|\script{H}_X(t)}$. 
Clearly $\expect{\tilde{X}(t)|\script{H}_X(t)} = 0$ for all $t$ and all $\script{H}_X(t)$.   Now define
$\script{H}_{\tilde{X}}(t)$ as the history of the $\tilde{X}(t)$ process: 
\[ \script{H}_{\tilde{X}}(t) \defequiv \{\tilde{X}(0), \ldots, \tilde{X}(t-1)\} \]
It is easy to see that 
conditioning on $\script{H}_{\tilde{X}}(t)$ is the same as conditioning on $\script{H}_{X}(t)$, because these provide
the same information.  Thus $\expect{\tilde{X}(t)|\script{H}_{\tilde{X}}(t)} = 0$ for all $t$ and all possible $\script{H}_{\tilde{X}}(t)$. 
To apply the result of Theorem \ref{lem:martingale}, 
we show that the second moment of $\tilde{X}(t)$ satisfies the condition (\ref{eq:vc2}). We
have for all $t$: 
\begin{eqnarray}
 \expect{\tilde{X}(t)^2} &=& \expect{(X(t) - \expect{X(t)|\script{H}_X(t)})^2} \nonumber \\
 &=&\expect{X(t)^2} + \expect{\expect{X(t)|\script{H}_X(t)}^2} - 2\expect{X(t)\expect{X(t)|\script{H}_X(t)}} \nonumber \\
 &\leq& \expect{X(t)^2} + \expect{\expect{X(t)^2|\script{H}_X(t)}} - 2\expect{X(t)\expect{X(t)|\script{H}_X(t)}} \label{eq:jensen}  \\
 &=& 2\expect{X(t)^2} - 2\expect{X(t)\expect{X(t)|\script{H}_X(t)}}\nonumber \\
 &\leq& 2\expect{X(t)^2} + 2\sqrt{\expect{X(t)^2}\expect{\expect{X(t)|\script{H}_X(t)}^2}} \label{eq:cs}  \\
 &\leq& 2\expect{X(t)^2} + 2\sqrt{\expect{X(t)^2}\expect{\expect{X(t)^2|\script{H}_X(t)}}} \nonumber \\
 &\leq& 2\expect{X(t)^2} + 2\sqrt{\expect{X(t)^2}}\sqrt{\expect{X(t)^2}} = 4\expect{X(t)^2} \nonumber 
 \end{eqnarray}
 where (\ref{eq:jensen}) follows by Jensen's inequality, and (\ref{eq:cs}) follows by the Cauchy-Schwartz inner product
 inequality.  It follows that:
\[ \sum_{t=1}^{\infty} \frac{\expect{\tilde{X}(t)^2}}{t^2} \leq \sum_{t=1}^{\infty} \frac{4\expect{X(t)^2}}{t^2} < \infty  \]
   Thus, the result of Theorem \ref{lem:martingale} holds for the process $\tilde{X}(t)$, and so: 
 \begin{equation*} 
  \lim_{t\rightarrow\infty} \frac{1}{t}\sum_{\tau=0}^{t-1} \tilde{X}(\tau) = 0 \: \: \: (w.p.1) 
  \end{equation*} 
 That is: 
 \begin{equation} 
  \lim_{t\rightarrow\infty} \frac{1}{t}\sum_{\tau=0}^{t-1} [X(\tau) - \expect{X(\tau)|\script{H}(\tau)}] = 0 \: \: \: (w.p.1) \label{eq:supermartingale-1} 
 \end{equation} 
Using the fact that $\expect{X(\tau)|\script{H}(\tau)} \leq B$ yields: 
\[ \frac{1}{t}\sum_{\tau=0}^{t-1}[X(\tau) - \expect{X(\tau)|\script{H}(\tau)}] \geq \frac{1}{t}\sum_{\tau=0}^{t-1}[X(\tau) - B] \]
Taking a $\limsup$ of the above as $t\rightarrow\infty$ and using (\ref{eq:supermartingale-1}) yields:
\[ \limsup_{t\rightarrow\infty} \frac{1}{t}\sum_{\tau=0}^{t-1} [X(\tau) - B] \leq 0 \: \: (w.p.1) \]
 This proves the result. 

\section*{Appendix B --- Proof of Claim 1 in Theorem \ref{thm:var-lyap}} 

Here we prove the Claim 1 needed in Theorem \ref{thm:var-lyap}.  Recall that $\delta(t) \defequiv L(\bv{Q}(t+1)) - L(\bv{Q}(t))$, 
where $L(\bv{Q}(t))$ is defined in (\ref{eq:lyapunov-function}) with any weights $w_k>0$. 
We prove Claim 1 with two lemmas. 

\begin{lem} \label{lem:finite-var} 
Suppose there is a finite constant $D>0$ such that for all $t$ and all possible
$\bv{Q}(t)$ we have: 
\[ \expect{d_k(t)^4|\bv{Q}(t)} \leq D \: \: \forall k \in \{1, \ldots, K\}, \forall t \in \{0, 1, 2, \ldots\}  \]
Further suppose that:  
\begin{equation} \label{eq:appc}
\sum_{t=1}^{\infty} \frac{\expect{Q_k(t)^2}}{t^2} < \infty 
\end{equation} 
Then: 
\[ \sum_{t=1}^{\infty}\frac{\expect{\delta(t)^2}}{t^2} < \infty \]
\end{lem} 
\begin{proof} 
We have: 
\begin{eqnarray*}
\delta(t) \defequiv  \sum_{k=1}^K\frac{w_k}{2}[Q_k(t+1)^2 - Q_k(t)^2] &=& \sum_{k=1}^K\frac{w_k}{2}[(Q_k(t) + d_k(t))^2 - Q_k(t)^2]\\
 &=&  \sum_{k=1}^K\frac{w_k}{2}[2Q_k(t)d_k(t) + d_k(t)^2]
 \end{eqnarray*}
 Thus: 
 \begin{eqnarray*}
 \expect{\delta(t)^2} =  \sum_{k=1}^K\sum_{i=1}^K\frac{w_kw_i}{4}\expect{(2Q_k(t)d_k(t)+d_k(t)^2)(2Q_i(t)d_i(t) + d_i(t)^2)} 
 \end{eqnarray*}
Further: 
\begin{eqnarray*}
\expect{(2Q_k(t)d_k(t) + d_k(t)^2)(2Q_i(t)d_i(t) + d_i(t)^2)} &=& 4\expect{Q_k(t)Q_i(t)d_k(t)d_i(t)} + \expect{d_k(t)^2d_i(t)^2} \\
&& +  2\expect{Q_k(t)d_k(t)d_i(t)^2} + 2\expect{Q_i(t)d_i(t)d_k(t)^2}  \\
&\leq& 4\sqrt{\expect{Q_k(t)^2d_k(t)^2}\expect{Q_i(t)^2d_i(t)^2}}  \\
&&+ \sqrt{\expect{d_k(t)^4}\expect{d_i(t)^4}} \\
&& + 2\sqrt{\expect{Q_k(t)^2d_k(t)^2}\expect{d_i(t)^4}} \\
&&+ 2\sqrt{\expect{Q_i(t)^2d_i(t)^2}\expect{d_k(t)^4}} 
\end{eqnarray*}
Because $\expect{d_k(t)^4|\bv{Q}(t)} \leq D$ for all possible $\bv{Q}(t)$, we have from iterated
expectations that for all $k \in \{1, \ldots, K\}$
\[ \expect{d_k(t)^4} \leq D \] 
Further, for all $k \in \{1, \ldots, K\}$ we have: 
\begin{eqnarray*}
 \expect{Q_k(t)^2d_k(t)^2} &=& \expect{\expect{Q_k(t)^2d_k(t)^2|\bv{Q}(t)}} \\
&=& \expect{Q_k(t)^2\expect{d_k(t)^2|\bv{Q}(t)}} \\
&\leq&\expect{Q_k(t)^2\sqrt{\expect{d_k(t)^4|\bv{Q}(t)}}} \\
&\leq& \expect{Q_k(t)^2D} \\
&\leq& D\expect{Q_{max}(t)^2} 
\end{eqnarray*}
where we define $Q_{max}(t)^2 \defequiv \max_{k\in\{1, \ldots, K\}} Q_k(t)^2$. Thus: 
\begin{eqnarray*}
\expect{(2Q_k(t)d_k(t) + d_k(t)^2)(2Q_i(t)d_i(t) + d_i(t)^2)} &\leq& 4D\expect{Q_{max}(t)^2} + D +4D \sqrt{\expect{Q_{max}(t)^2}} \\
&\leq& D_1\expect{Q_{max}(t)^2} + D_2
\end{eqnarray*}
for some positive constants $D_1$,  $D_2$. Thus: 
\begin{eqnarray*}
 \expect{\delta(t)^2} &\leq& (D_1\expect{Q_{max}(t)^2} + D_2)\sum_{k=1}^K\sum_{i=1}^K\frac{w_kw_i}{4} \\
 &\leq& D_3 + D_4\sum_{k=1}^K\expect{Q_k(t)^2} 
 \end{eqnarray*}
 for some positive constants $D_3$, $D_4$.  Thus: 
 \begin{eqnarray}
\sum_{t=1}^{\infty} \frac{\expect{\delta(t)^2}}{t^2} \leq \sum_{t=1}^{\infty} \frac{D_3}{t^2} + D_4\sum_{k=1}^K\sum_{t=1}^{\infty} \frac{\expect{Q_k(t)^2}}{t^2} < \infty
 \end{eqnarray}
\end{proof} 

Now fix any constants $V$, $B$, $p^*$, $\epsilon$, and recall that $\beta(t)$ is defined: 
\[ \beta(t) \defequiv Vp(t) - B - Vp^* + \epsilon\sum_{k=1}^K|Q_k(t)| \]

\begin{lem} \label{lem:finite-var2} 
Suppose that for all $k \in \{1, \ldots, K\}$ we have:  
\begin{eqnarray} 
\sum_{t=1}^{\infty} \frac{\expect{Q_k(t)^2}}{t} &<& \infty \label{eq:zoo1} \\
\sum_{t=1}^{\infty} \frac{\expect{p(t)^2}}{t^2} &<& \infty \label{eq:zoo2} 
\end{eqnarray} 
Then: 
\[ \sum_{t=1}^{\infty}\frac{\expect{\beta(t)^2}}{t^2} < \infty \]
\end{lem} 

Note that Lemmas \ref{lem:finite-var} and \ref{lem:finite-var2} together prove Claim 1.  It
remains only to prove Lemma \ref{lem:finite-var2}. 

\begin{proof} (Lemma \ref{lem:finite-var2})
We have: 
\begin{eqnarray*}
\expect{\beta(t)^2} &=& \expect{(Vp(t) - B - Vp^*)^2} + \epsilon^2\sum_{k=1}^K\sum_{i=1}^K\expect{|Q_k(t)||Q_i(t)|} \\
&& + \epsilon\sum_{k=1}^{K}\expect{(Vp(t) - B - Vp^*)|Q_k(t)|}  \\
&\leq& \expect{(Vp(t) - B - Vp^*)^2}  + \epsilon^2\sum_{k=1}^K\sum_{i=1}^K\sqrt{\expect{Q_k(t)^2}\expect{Q_i(t)^2}} \\
&& + \epsilon\sum_{k=1}^K\sqrt{\expect{(Vp(t) - B - Vp^*)^2}\expect{Q_k(t)^2}} 
\end{eqnarray*}
However, because $|ab| \leq \frac{1}{2}[a^2 + b^2]$ for all real numbers $a,b$,  we have: 
\[ \sqrt{\expect{(Vp(t) - B - Vp^*)^2}\expect{Q_k(t)^2}} \leq \frac{1}{2}\expect{(Vp(t) - B - Vp^*)^2}+\frac{1}{2}\expect{Q_k(t)^2}\]
Thus: 
\begin{eqnarray*}
\expect{\beta(t)^2} &\leq& \expect{(Vp(t) - B - Vp^*)^2}  + \epsilon^2\sum_{k=1}^K\sum_{i=1}^K\sqrt{\expect{Q_k(t)^2}\expect{Q_i(t)^2}} \\
&& + \frac{\epsilon}{2}\sum_{k=1}^K\expect{(Vp(t) - B - Vp^*)^2}+\frac{\epsilon}{2}\sum_{k=1}^K\expect{Q_k(t)^2} \\
&\leq& (1+\epsilon K/2)\expect{(Vp(t) - B - Vp^*)^2} + (\epsilon^2 K^2 + \epsilon K/2)\expect{Q_{max}(t)^2} 
\end{eqnarray*}
where we define $Q_{max}(t)^2 \defequiv \max_{k\in\{1, \ldots, K\}} Q_k(t)^2$.   It follows that there are finite
constants $D_1, D_2, D_3$ such that: 
\[ \expect{\beta(t)^2} \leq D_1 + D_2\expect{p(t)^2} + D_3\expect{Q_{max}(t)^2} \]
Because $Q_{max}(t)^2 \leq \sum_{k=1}^KQ_k(t)^2$, we have: 
\[ \expect{\beta(t)^2} \leq D_1 +D_2\expect{p(t)^2} + D_3\sum_{k=1}^K\expect{Q_k(t)^2} \]
Thus, from (\ref{eq:zoo1})-(\ref{eq:zoo2}) we have: 
\begin{eqnarray*}
\sum_{t=1}^{\infty} \frac{\expect{\beta(t)^2}}{t^2} &\leq& \sum_{t=1}^{\infty}\frac{D_1}{t^2} + D_2\sum_{t=1}^{\infty}\frac{\expect{p(t)^2}}{t^2} + D_3\sum_{k=1}^K \sum_{t=1}^{\infty}\frac{\expect{Q_k(t)^2}}{t^2} < \infty
\end{eqnarray*}
which proves the result. 
\end{proof}

\section*{Appendix C --- Proof of Theorem \ref{thm:quad-lyap} parts (a) and (b)}

\begin{proof} (Theorem \ref{thm:quad-lyap} part (a)) 
The proof closely follows a similar result derived for exponential Lyapunov functions with 
deterministically bounded queue changes in \cite{longbo-lagrange}. 
From (\ref{eq:quad-lyap-drift}) we have: 
\[ \expect{L(\bv{Q}(t+1))| \bv{Q}(t)}  \leq L(\bv{Q}(t)) + \tilde{B} - \epsilon \sum_{k=1}^K |Q_k(t)| \]
Therefore: 
\begin{eqnarray*}
 \expect{\norm{\bv{Q}(t+1)}^2|\bv{Q}(t)} &\leq& \norm{\bv{Q}(t)}^2 + \tilde{B} - \epsilon \sum_{k=1}^K|Q_k(t)| \\
 &\leq& \norm{\bv{Q}(t)}^2 + \tilde{B} - \frac{\epsilon\sqrt{2}}{\sqrt{w_{max}}}\sum_{k=1}^K\frac{\sqrt{w_k}}{\sqrt{2}}|Q_k(t)| \\
 &\leq& \norm{\bv{Q}(t)}^2 + \tilde{B} - \frac{\epsilon\sqrt{2}}{\sqrt{w_{max}}}\norm{\bv{Q}(t)}  \\
 &=& \norm{\bv{Q}(t)}^2 + \tilde{B} -4c\norm{\bv{Q}(t)}  \\
 \end{eqnarray*}
 where $w_{max} \defequiv \max_{k \in \{1, \ldots, K\}} w_k$ and $c \defequiv \epsilon\sqrt{2}/(4\sqrt{w_{max}})$. 
 The third
 inequality above follows by (\ref{eq:triangle}). 
Now suppose that $\norm{\bv{Q}(t)} \geq \tilde{B}/(2c)$.  It follows that: 
\begin{eqnarray*}
  \expect{\norm{\bv{Q}(t+1)}^2|\bv{Q}(t)} &\leq& \norm{\bv{Q}(t)}^2 + \tilde{B} - 2c\norm{\bv{Q}(t)} - 2c\norm{\bv{Q}(t)} \\
&\leq& \norm{\bv{Q}(t)}^2  - 2c\norm{\bv{Q}(t)}  \\
&\leq& \norm{\bv{Q}(t)}^2  - 2c\norm{\bv{Q}(t)}  + c^2 \\
&=& (\norm{\bv{Q}(t)} - c)^2
\end{eqnarray*}
However, we have by Jensen's inequality: 
\begin{eqnarray*}
 \expect{\norm{\bv{Q}(t+1)}|\bv{Q}(t)}^2 \leq  \expect{\norm{\bv{Q}(t+1)}^2|\bv{Q}(t)}
\end{eqnarray*}
Therefore: 
\[ \expect{\norm{\bv{Q}(t+1)}|\bv{Q}(t)}^2 \leq (\norm{\bv{Q}(t)} - c)^2 \]
Assume now that $\norm{\bv{Q}(t)} \geq \max[\tilde{B}/(2c), c]$, so that we have both that 
$\norm{\bv{Q}(t)} - c \geq 0$ and 
$\norm{\bv{Q}(t)} \geq \tilde{B}/(2c)$. 
Taking square roots of the above inequality then proves that whenever $\norm{\bv{Q}(t)} \geq \max[\tilde{B}/(2c), c]$ we have: 
\[ \expect{\norm{\bv{Q}(t+1)}|\bv{Q}(t)} \leq \norm{\bv{Q}(t)} - c \]
Defining $a \defequiv \max[\tilde{B}/(2c), c]$ proves part (a).
\end{proof} 

\begin{proof} (Theorem \ref{thm:quad-lyap} part (b)) 
We have $\bv{Q}(t+1) = \bv{Q}(t) + \bv{d}(t)$, where $\bv{d}(t)  \defequiv (d_1(t), \ldots, d_K(t))$.
Define $\gamma(t) \defequiv \norm{\bv{Q}(t+1)} - \norm{\bv{Q}(t)}$. 
Then $|\gamma(t)| \leq \norm{\bv{d}(t)}$ (by (\ref{eq:t2})), and we have: 
\begin{eqnarray}
\norm{\bv{Q}(t+1)}^4 &=&(\norm{\bv{Q}(t)} + \gamma(t))^4 \nonumber \\
&=& \norm{\bv{Q}(t)}^4 + 4\norm{\bv{Q}(t)}^3\gamma(t) + 6\norm{\bv{Q}(t)}^2\gamma(t)^2 \nonumber \\
&& + 4\norm{\bv{Q}(t)}\gamma(t)^3 + \gamma(t)^4 \label{eq:zaz} 
\end{eqnarray}
However, note by part (a) that $\expect{\gamma(t)|\bv{Q}(t)} \leq -c$ whenever $\norm{\bv{Q}(t)} \geq a$
(for some constants $c>0$, $a>0$).  Thus: 
\[ 4\norm{\bv{Q}(t)}^3\expect{\gamma(t)|\bv{Q}(t)} \leq  \left\{ \begin{array}{ll}
                          -4c\norm{\bv{Q}(t)}^3  &\mbox{ if $\norm{\bv{Q}(t)} \geq a$} \\
                              4a^3\expect{\norm{\bv{d}(t)}|\bv{Q}(t)} & \mbox{ otherwise} 
                            \end{array}
                                 \right.\]
Hence: 
\begin{eqnarray*}
4\norm{\bv{Q}(t)}^3\expect{\gamma(t)|\bv{Q}(t)} &\leq& -4c\norm{\bv{Q}(t)}^3 + 4a^3\expect{\norm{\bv{d}(t)}|\bv{Q}(t)} + 4ca^3
\end{eqnarray*}
Taking conditional expectations of (\ref{eq:zaz}) and substituting the above yields: 
\begin{eqnarray}
\expect{\norm{\bv{Q}(t+1)}^4|\bv{Q}(t)} &\leq& \norm{\bv{Q}(t)}^4 - 4c\norm{\bv{Q}(t)}^3 + 4a^3\expect{\norm{\bv{d}(t)}|\bv{Q}(t)} + 4ca^3 \nonumber \\
&& + 6\norm{\bv{Q}(t)}^2\expect{\norm{\bv{d}(t)}^2|\bv{Q}(t)} + \nonumber \\
&& 4\norm{\bv{Q}(t)}\expect{\norm{\bv{d}(t)}^3|\bv{Q}(t)} + \expect{\norm{\bv{d}(t)}^4|\bv{Q}(t)}  \label{eq:great} 
\end{eqnarray}
Because $\norm{\bv{d}(t)} \leq g \sum_{k=1}^K|d_k(t)|$ (where $g \defequiv \sqrt{w_{max}/2}$, with $w_{max} \defequiv \max_{k\in\{1, \ldots, K\}} w_k$),  
we have: 
\begin{eqnarray*}
\expect{\norm{\bv{d}(t)}^4|\bv{Q}(t)} \leq g^4\sum_{k=1}^K\sum_{i=1}^K\sum_{j=1}^K\sum_{l=1}^K\expect{|d_k(t)||d_i(t)||d_j(t)||d_l(t)| | \bv{Q}(t)}
\end{eqnarray*}
However, by repeated application of Cauchy-Schwartz and the fact that $\expect{d_k(t)^4|\bv{Q}(t)} \leq D$, we have: 
\[ \expect{|d_k(t)||d_i(t)||d_j(t)||d_l(t)| | \bv{Q}(t)} \leq D \]
Thus: 
\begin{equation} \label{eq:great2}
 \expect{\norm{\bv{d}(t)}^4 | \bv{Q}(t)} \leq g^4 K^4D  
 \end{equation} 
Further, by Jensen's inequality: 
\begin{eqnarray}
\expect{\norm{\bv{d}(t)}^3|\bv{Q}(t)} &\leq& \expect{\norm{\bv{d}(t)}^4|\bv{Q}(t)}^{3/4} \leq D^{3/4} \label{eq:great3}  \\
\expect{\norm{\bv{d}(t)}^2|\bv{Q}(t)} &\leq& \expect{\norm{\bv{d}(t)}^4|\bv{Q}(t)}^{1/2} \leq D^{1/2} \label{eq:great4} \\
\expect{\norm{\bv{d}(t)}|\bv{Q}(t)} &\leq& \expect{\norm{\bv{d}(t)}^4|\bv{Q}(t)}^{1/4} \leq D^{1/4}  \label{eq:great5} 
\end{eqnarray}
Substituting (\ref{eq:great2})-(\ref{eq:great5}) into (\ref{eq:great}) yields:  
\begin{eqnarray}
\expect{\norm{\bv{Q}(t+1)}^4|\bv{Q}(t)} - \norm{\bv{Q}(t)}^4  &\leq&  - 4c\norm{\bv{Q}(t)}^3 + 4a^3D^{1/4} + 4ca^3\nonumber \\
&& + 6\norm{\bv{Q}(t)}^2D^{1/2} + 4\norm{\bv{Q}(t)}D^{3/4} + D \label{eq:great6} 
\end{eqnarray}
Because the term $-4c\norm{\bv{Q}(t)}^3$ is the dominant term on the right-hand-side above (for $\norm{\bv{Q}(t)}$ large), 
 there must be a constant $b_1>0$ such that: 
 \[ -2c\norm{\bv{Q}(t)}^3 +  4a^3D^{1/4} + 4ca^3 + 6\norm{\bv{Q}(t)}^2D^{1/2} + 4\norm{\bv{Q}(t)}D^{3/4} + D \leq 0 \]
 whenever $\norm{\bv{Q}(t)} \geq b_1$.  Thus, the right-hand-side of (\ref{eq:great6}) is less than or equal 
 to $-2c\norm{\bv{Q}(t)}^3$ whenever $\norm{\bv{Q}(t)} \geq b_1$, and is less than or equal to 
 $4a^3D^{1/4} + 4ca^3 + 6b_1^2D^{1/2} + 4b_1D^{3/4} + D$ otherwise.  It follows that there
 are constants $b_2>0$, $c>0$ such that for all $t$ and all $\bv{Q}(t)$ we have: 
 \[ \expect{\norm{\bv{Q}(t+1)}^4|\bv{Q}(t)} - \norm{\bv{Q}(t)}^4 \leq b_2 - 2c\norm{\bv{Q}(t)}^3 \]
 Taking expectations of the above yields: 
 \[ \expect{\norm{\bv{Q}(t+1)}^4} - \expect{\norm{\bv{Q}(t)}^4} \leq b_2 - 2c\expect{\norm{\bv{Q}(t)}^3} \]
 Summing the above over $t \in \{0, \ldots, M-1\}$ and dividing by $M$ yields: 
 \[ \frac{\expect{\norm{\bv{Q}(M)}^4} - \expect{\norm{\bv{Q}(0)}^4}}{M} \leq b_2 - 2c\frac{1}{M}\sum_{t=0}^{M-1}\expect{\norm{\bv{Q}(t)}^3} \]
 Rearranging terms and using the fact that $\norm{\bv{Q}(M)}^4 \geq 0$ yields: 
 \[ \frac{1}{M} \sum_{t=0}^{M-1} \expect{\norm{\bv{Q}(t)}^3} \leq \frac{b_2}{2c} + \frac{\expect{\norm{\bv{Q}(0)}^4}}{2cM}  \leq
  \frac{b_2}{2c} + \frac{\expect{\norm{\bv{Q}(0)}^4}}{2c}
 \]
 This completes the proof of  part (b). 
\end{proof}

\section*{Appendix D} 

\begin{lem} \label{lem:appd} Suppose $\{x_i\}_{i=1}^{\infty}$ is an infinite 
sequence of non-negative real numbers such that there are constants $C>0$ and $0 \leq \theta < 1$
such that: 
\[ \sum_{i=1}^M x_i \leq CM^{1+\theta} \: \: \forall M \in \{1, 2, 3, \ldots\} \]
Then: 
\[ \sum_{i=1}^{\infty} \frac{x_i}{i^2} < \infty \]
\end{lem} 
\begin{proof} 
For $M \in \{1, 2, 3, \ldots\}$, define $\phi(M)$ as: 
\[ \phi(M) \defequiv \frac{1}{M^2}\sum_{i=1}^Mx_i \]
Then clearly: 
\begin{equation} \label{eq:boundphi} 
 \phi(M) \leq \frac{C}{M^{1-\theta}} \: \: \forall M \in \{1, 2, 3, \ldots\} 
 \end{equation} 
 On the other hand, from the definition of $\phi(M)$ we have for all $M \in \{1, 2,3, \ldots\}$: 
 \[ \phi(M+1) = \phi(M)\frac{M^2}{(M+1)^2} + \frac{x_{M+1}}{(M+1)^2} \]
So: 
\[ \phi(M+1) = \phi(M)\left[1 - \frac{2}{M+1} + \frac{1}{(M+1)^2}\right] + \frac{x_{M+1}}{(M+1)^2} \]
Thus: 
\begin{eqnarray*}
 \frac{x_{M+1}}{(M+1)^2} &=& \phi(M+1) - \phi(M) + \frac{2\phi(M)}{M+1} - \frac{\phi(M)}{(M+1)^2}\\
&\leq& \phi(M+1) - \phi(M) + \frac{2\phi(M)}{M+1} 
\end{eqnarray*}
where the final inequality holds because $\phi(M) \geq 0$. 
Summing the above over $M \in \{1, \ldots, G\}$ for some positive integer $G$ yields: 
\begin{eqnarray*}
\sum_{M=1}^{G} \frac{x_{M+1}}{(M+1)^2} &\leq& \phi(G+1) - \phi(1) + 2\sum_{M=1}^G\frac{\phi(M)}{M+1} \\
\end{eqnarray*}
Because $\phi(1) = x_1$, rearranging the above yields: 
\begin{eqnarray} 
 \sum_{M=1}^{G+1} \frac{x_M}{M^2} &\leq& \phi(G+1) + 2\sum_{M=1}^G \frac{\phi(M)}{M+1} \nonumber \\
 &\leq& \frac{C}{(G+1)^{1-\theta}} + 2\sum_{M=1}^G \frac{C}{M^{1-\theta}(M+1)} \label{eq:appdlast} 
 \end{eqnarray}
 where (\ref{eq:appdlast}) follows from (\ref{eq:boundphi}).  Because $\theta < 1$, the first term on the right-hand-side
 of (\ref{eq:appdlast}) goes to $0$ as $G\rightarrow\infty$, and the second term is a summable series and hence is 
 less than a bounded constant as $G\rightarrow\infty$.  Thus: 
 \[ \lim_{G\rightarrow\infty} \sum_{M=1}^{G+1} \frac{x_M}{M^2} < \infty \]
\end{proof} 

\section*{Appendix E --- Proof of Theorem \ref{thm:rs-main} on Rate Stability} 
 
 We prove Theorem \ref{thm:rs-main} with the help of two preliminary lemmas. 
 Let $Q(t)$ be a non-negative stochastic process defined over $t \in \{0, 1, 2, \ldots\}$. 
 Fix $\delta>0$, and for each non-negative integer $n$ define $t_n(\delta)$ by: 
 \begin{equation} \label{eq:t-n-delta} 
  t_n(\delta) \defequiv \lceil n^{1+\delta}\rceil 
  \end{equation} 
  where $\lceil x \rceil$ represents the smallest integer greater than or equal to $x$. 
 The sequence $\{t_n(\delta)\}_{n=0}^{\infty}$ is a (sparse) 
 subsequence of the non-negative integers that increases super-linearly with $n$.  
 Lemma \ref{lem:rs1} below shows that if $\expect{Q(t)^2}$ grows at most linearly with $t$,  then
 $Q(t)$ is rate stable when sampled over the subsequence $\{t_n(\delta)\}_{n=0}^{\infty}$. 
 We note that rate stability over this sparse sampling is not as strong as ordinary rate stability.  This is 
 because $Q(t)/t$ may 
 not converge to zero, even though it converges to $0$ over the sparse sampling. However, Lemma \ref{lem:rs2} below
 shows that rate stability over the sparse sampling, together with an additional second moment bound on changes
 in $Q(t)$, is sufficient to ensure ordinary rate stability. 
  
 \begin{lem} \label{lem:rs1} Suppose  there is a finite constant $C>0$ and a positive integer $t^*$ such that: 
 \begin{equation*} 
 \expect{Q(t)^2} \leq Ct \: \: \forall t \geq t^*
 \end{equation*}
  Then for any $\delta>0$,  $Q(t)$ is rate stable when sampled over the subsequence of times $\{t_n(\delta)\}_{n=0}^{\infty}$. That is: 
  \begin{equation*} 
  \lim_{n\rightarrow\infty} \frac{Q(t_n(\delta))}{t_n(\delta)} = 0 \: \: (w.p.1) 
  \end{equation*} 
  \end{lem} 
  \begin{proof} 
  Fix $\epsilon>0$.  It suffices to show that: 
  \begin{equation} 
   \lim_{M\rightarrow\infty} Pr[\cup_{n \geq M} \{Q(t_n(\delta))/t_n(\delta) > \epsilon\}] = 0 \label{eq:to-prove-rs} 
  \end{equation} 
  To this end, note by the Markov inequality that for any slot $t\geq t^*$: 
  \[ Pr[Q(t)/t > \epsilon] = Pr[Q(t)^2 > \epsilon^2t^2] \leq \frac{\expect{Q(t)^2}}{\epsilon^2t^2} \leq \frac{C}{\epsilon^2t} \]
  Substituting $t = t_n(\delta)$ into the above inequality (assuming that $t_n(\delta)\geq t^*$) yields: 
  \[ Pr[Q(t_n(\delta))/t_n(\delta) > \epsilon] \leq \frac{C}{\epsilon^2 t_n(\delta)} \leq \frac{C}{\epsilon^2 n^{(1+\delta)}} \]
  Therefore, by the union bound, we have for any positive integer $M$ such that $t_M(\delta) \geq t^*$:
  \begin{eqnarray*}
   0 \leq Pr[\cup_{n\geq M} \{Q(t_n(\delta))/t_n(\delta) > \epsilon\}] &\leq& \sum_{n=M}^{\infty} Pr[Q(t_n(\delta))/t_n(\delta)>\epsilon] \\
  &\leq& \sum_{n=M}^{\infty} \frac{C}{\epsilon^2 n^{(1+\delta)}} < \infty 
  \end{eqnarray*}
  Thus, the probability on the left-hand-side of the above chain of inequalities 
  is bounded by the tail of a convergent series, and so (\ref{eq:to-prove-rs}) holds.   
  \end{proof} 
  
 \begin{lem} \label{lem:rs2} Suppose  there is a finite constant $C>0$ and a positive integer $t^*$ such that: 
 \begin{equation*}
 \expect{Q(t)^2} \leq Ct \: \: \forall t \geq t^*
 \end{equation*}
  Further suppose there is a finite constant $D>0$ such that for all $t\in\{0, 1, 2, \ldots\}$ we have: 
 \[ \expect{(Q(t+1) - Q(t))^2} \leq D \]
 Then $Q(t)$ is rate stable. 
 \end{lem} 
 \begin{proof} 
 Fix a value $\delta$ such that $0 < \delta < 1$ and $\delta + (3/4)(1+\delta ) < 1$.  
 By Lemma \ref{lem:rs1} we know that $Q(t)$ is rate stable when sampled over times $\{t_n(\delta)\}_{n=0}^{\infty}$, where $t_n(\delta)$ is 
 defined in (\ref{eq:t-n-delta}). For
 simplicity of notation, below we write ``$t_n$'' in replacement for ``$t_n(\delta)$.''
 Thus, $t_n \defequiv \lceil n^{(1+\delta)}\rceil$, and:
 \[ \lim_{n\rightarrow\infty} \frac{Q(t_n)}{t_n} = 0 \: \: (w.p.1) \]

 Now note by the Markov inequality that for all $t \geq0$: 
 \[ Pr[|Q(t+1)-Q(t)| \geq t^{3/4}] = Pr[(Q(t+1)-Q(t))^2 \geq t^{3/2}] \leq \frac{D}{t^{3/2}} \]
 Thus, for any integer $M>0$: 
 \[ Pr[\cup_{t\geq M} \{|Q(t+1) - Q(t)| \geq t^{3/4}\}] \leq \sum_{t=M}^{\infty} \frac{D}{t^{3/2}} < \infty \]
 Thus: 
 \[ \lim_{M\rightarrow\infty} Pr[\cup_{t\geq M} \{|Q(t+1) - Q(t)| \geq t^{3/4}\}] = 0 \]
 It follows that, with probability 1, there is some positive random integer $K$ such that 
 $|Q(t+1) - Q(t)| < t^{3/4}$ for all $t \geq K$.  
 
 Now for any integer $t>0$, define $n(t)$ as the integer such that $t_{n(t)} \leq t < t_{n(t)+1}$.  Then for 
 any $t>0$ such that $t_{n(t)} \geq K$, we have: 
 \begin{eqnarray*}
 Q(t) &\leq& Q(t_{n(t)}) + [t_{n(t)+1} - t_{n(t)}]t_{n(t)+1}^{3/4} \\
 &\leq&  Q(t_{n(t)}) + [t_{n(t)+1} - t_{n(t)}][(n(t)+1)^{(3/4)(1+\delta)} + 1]
 \end{eqnarray*}
 Thus: 
  \begin{eqnarray}
 \frac{Q(t)}{t} &\leq& \frac{Q(t_{n(t)}) + [t_{n(t)+1} - t_{n(t)}][(n(t)+1)^{(3/4)(1+\delta)}+1]}{t_{n(t)}} \label{eq:xybaz} 
 \end{eqnarray}
 On the other hand, for any $n>0$ we have by a Taylor expansion
 \begin{eqnarray*}
  t_{n+1}  &\leq& 1 + (n+1)^{1+\delta} \\
  &\leq& 1 + n^{1+\delta} + (1+\delta)n^{\delta} + \frac{(1+\delta)\delta}{2} n^{\delta-1} \\
  &\leq& a + n^{1+\delta} +(1+\delta)n^{\delta}
  \end{eqnarray*}
  where $a\defequiv 1 + (1+\delta)\delta/2$. Thus, for any $n(t)>0$ we have: 
  \[ t_{n(t)+1} - t_{n(t)} \leq t_{n(t)+1} - n(t)^{1+\delta} \leq a + (1+\delta)n(t)^{\delta}  \]
  Using this in (\ref{eq:xybaz}) yields: 
    \begin{eqnarray}
 0 \leq \frac{Q(t)}{t} &\leq& \frac{Q(t_{n(t)}) + [a + (1+\delta)n(t)^{\delta} ][(n(t)+1)^{(3/4)(1+\delta)}+1]}{t_{n(t)}} \label{eq:xybaz2} \\
 &\leq& \frac{Q(t_{n(t)})}{t_{n(t)}} + \frac{a[(n(t)+1)^{3/4(1+\delta)}+1]}{n(t)^{1+\delta}} + \frac{(1+\delta)n(t)^{\delta}[(n(t)+1)^{(3/4)(1+\delta)}+1]}{n(t)^{1+\delta}} 
 \end{eqnarray}
 Taking limits and using the fact that $Q(t_{n(t)})/t_{n(t)}\rightarrow 0$ with probability 1, and the fact that $\delta + (3/4)(1+\delta)<1$, 
 yields: 
 \[ 0 \leq \lim_{t\rightarrow\infty} Q(t)/t \leq 0 \: \: (w.p.1) \]
\end{proof} 

We now prove Theorem \ref{thm:rs-main}. 
Let $\bv{Q}(t) = (Q_1(t), \ldots, Q_K(t))$ be a 
stochastic vector defined over $t \in \{0, 1, 2, \ldots\}$.  Assume $\bv{Q}(t)$ has
real-valued entries. Define the quadratic
Lyapunov function $L(\bv{Q}(t))$ as in (\ref{eq:lyapunov-function}) and define the drift
$\Delta(\script{H}(t))$ as in (\ref{eq:lyapunov-drift}). 
Suppose there is a finite constant $B>0$ such that for all $\tau \in \{0, 1, 2, \ldots\}$ and all possible $\script{H}(\tau)$, 
we have: 
\begin{equation} \label{eq:dpp-apprs} 
\Delta(\script{H}(\tau)) \leq B 
\end{equation} 

\begin{proof} (Theorem \ref{thm:rs-main} part (a)) 
Assume that $\expect{L(\bv{Q}(0))} < \infty$. 
Fix a slot $\tau\geq0$. Taking expectations of (\ref{eq:dpp-apprs}) yields: 
\[ \expect{L(\bv{Q}(\tau+1))} - \expect{L(\bv{Q}(\tau))} \leq B \]
Summing the above over $\tau \in \{0, 1, \ldots, t-1\}$ for some integer $t>0$ yields: 
\[ \expect{L(\bv{Q}(t))} - \expect{L(\bv{Q}(0))} \leq Bt \]
Substituting the definition of $L(\bv{Q}(t))$ in (\ref{eq:lyapunov-function}) into the above inequality yields: 
\begin{equation} \label{eq:thm-app-rs} 
 \frac{1}{2}\sum_{k=1}^Kw_k\expect{Q_k(t)^2} \leq Bt + \expect{L(\bv{Q}(0))} 
 \end{equation} 

It follows from (\ref{eq:thm-app-rs}) that for each $k \in \{1, \ldots, K\}$: 
\begin{equation} \label{eq:alfa1} 
\expect{|Q_k(t)|}^2 \leq \expect{Q_k(t)^2} \leq \frac{2Bt + 2\expect{L(\bv{Q}(0))}}{w_k}
\end{equation} 
and so: 
\[ \expect{|Q_k(t)|} \leq \sqrt{2Bt/w_k+ 2\expect{L(\bv{Q}(0))}/w_k} \]
Dividing the above by $t$ and taking limits as $t\rightarrow\infty$ shows
that $Q_k(t)$ is mean rate stable, 
proving part (a). 
\end{proof} 

\begin{proof} (Theorem \ref{thm:rs-main} part (b)) 
First assume that $\bv{Q}(0)$ is a given finite constant (with probability 1), so that $\expect{L(\bv{Q}(0))} = L(\bv{Q}(0))$. 
We have from (\ref{eq:alfa1}) that for all $t\geq 1$ and all $k \in \{1, \ldots, K\}$:  
\[ \expect{Q_k(t)^2} \leq \frac{[2B+ 2L(\bv{Q}(0))]t}{w_k}\]
Furthermore, it can be shown that $\expect{(Q_k(t+1) - Q_k(t))^2} \leq D$ implies 
$\expect{(|Q_k(t+1)| - |Q_k(t)|)^2} \leq D$. 
Thus,  the conditions required to apply Lemma \ref{lem:rs2} hold (using $Q(t) = |Q_k(t)|$, 
$t^* = 1$ and $C = [2B + 2L(\bv{Q}(0))]/w_k$). 
Then Lemma \ref{lem:rs2} ensures $|Q_k(t)|$ is rate stable for all $k \in \{1, \ldots, K\}$,
and hence $Q_k(t)$ is rate stable for all $k \in \{1, \ldots, K\}$. The above holds whenever the initial condition
$\bv{Q}(0)$ is any given finite constant, and hence it holds whenever $\bv{Q}(0)$ is finite with probability 1. 
\end{proof} 

\bibliographystyle{unsrt}
\bibliography{../../../latex-mit/bibliography/refs}

\begin{thebibliography}{10}

\bibitem{now}
L.~Georgiadis, M.~J. Neely, and L.~Tassiulas.
\newblock Resource allocation and cross-layer control in wireless networks.
\newblock {\em Foundations and Trends in Networking}, vol. 1, no. 1, pp. 1-149,
  2006.

\bibitem{neely-thesis}
M.~J. Neely.
\newblock {\em Dynamic Power Allocation and Routing for Satellite and Wireless
  Networks with Time Varying Channels}.
\newblock PhD thesis, Massachusetts Institute of Technology, LIDS, 2003.

\bibitem{neely-fairness-infocom05}
M.~J. Neely, E.~Modiano, and C.~Li.
\newblock Fairness and optimal stochastic control for heterogeneous networks.
\newblock {\em Proc. IEEE INFOCOM}, March 2005.

\bibitem{neely-energy-it}
M.~J. Neely.
\newblock Energy optimal control for time varying wireless networks.
\newblock {\em IEEE Transactions on Information Theory}, vol. 52, no. 7, pp.
  2915-2934, July 2006.

\bibitem{neely-asilomar08}
M.~J. Neely and R.~Urgaonkar.
\newblock Opportunism, backpressure, and stochastic optimization with the
  wireless broadcast advantage.
\newblock {\em Asilomar Conference on Signals, Systems, and Computers, Pacific
  Grove, CA}, Oct. 2008.

\bibitem{stolyar-greedy}
A.~Stolyar.
\newblock Maximizing queueing network utility subject to stability: Greedy
  primal-dual algorithm.
\newblock {\em Queueing Systems}, vol. 50, no. 4, pp. 401-457, 2005.

\bibitem{atilla-fairness-ton}
A.~Eryilmaz and R.~Srikant.
\newblock Fair resource allocation in wireless networks using
  queue-length-based scheduling and congestion control.
\newblock {\em IEEE/ACM Transactions on Networking}, vol. 15, no. 6, pp.
  1333-1344, Dec. 2007.

\bibitem{neely-utility-delay-jsac}
M.~J. Neely.
\newblock Super-fast delay tradeoffs for utility optimal fair scheduling in
  wireless networks.
\newblock {\em IEEE Journal on Selected Areas in Communications, Special Issue
  on Nonlinear Optimization of Communication Systems}, vol. 24, no. 8, pp.
  1489-1501, Aug. 2006.

\bibitem{neely-energy-delay-it}
M.~J. Neely.
\newblock Optimal energy and delay tradeoffs for multi-user wireless downlinks.
\newblock {\em IEEE Transactions on Information Theory}, vol. 53, no. 9, pp.
  3095-3113, Sept. 2007.

\bibitem{tutorial-lin}
X.~Lin, N.~B. Shroff, and R.~Srikant.
\newblock A tutorial on cross-layer optimization in wireless networks.
\newblock {\em IEEE Journal on Selected Areas in Communications, Special Issue
  on Nonlinear Optimization of Communication Systems}, vol. 14, no. 8, Aug.
  2006.

\bibitem{atilla-primal-dual-jsac}
A.~Eryilmaz and R.~Srikant.
\newblock Joint congestion control, routing, and mac for stability and fairness
  in wireless networks.
\newblock {\em IEEE Journal on Selected Areas in Communications, Special Issue
  on Nonlinear Optimization of Communication Systems}, vol. 14, pp. 1514-1524,
  Aug. 2006.

\bibitem{chiang-sno}
Y.~Yi and M.~Chiang.
\newblock Stochastic network utility maximization: A tribute to {K}elly's paper
  published in this journal a decade ago.
\newblock {\em European Transactions on Telecommunications}, vol. 19, no. 4,
  pp. 421-442, June 2008.

\bibitem{primal-dual-cmu}
Q.~Li and R.~Negi.
\newblock Scheduling in wireless networks under uncertainties: A greedy
  primal-dual approach.
\newblock {\em Arxiv Technical Report: arXiv:1001:2050v2}, June 2010.

\bibitem{neely-mwl-arxiv}
M.~J. Neely.
\newblock Max weight learning algorithms with application to scheduling in
  unknown environments.
\newblock {\em arXiv:0902.0630v1}, Feb. 2009.

\bibitem{williams-martingale}
D.~Williams.
\newblock {\em Probability with Martingales}.
\newblock Cambridge Mathematical Textbooks, Cambridge University Press, 1991.

\bibitem{meyn-book}
S.~P. Meyn and R.~L. Tweedie.
\newblock {\em Markov Chains and Stochastic Stability}.
\newblock Springer-Verlag, London, 1993.

\bibitem{stability}
M.~J. Neely.
\newblock Stability and capacity regions for discrete time queueing networks.
\newblock {\em ArXiv Technical Report: arXiv:1003.3396v1}, March 2010.

\bibitem{neely-universal-scheduling-cdc2010}
M.~J. Neely.
\newblock Universal scheduling for networks with arbitrary traffic, channels,
  and mobility.
\newblock {\em Proc. IEEE Conf. on Decision and Control (CDC)}, Atlanta, GA,
  Dec. 2010.

\bibitem{lee-stochastic-scheduling}
J.~W. Lee, R.~R. Mazumdar, and N.~B. Shroff.
\newblock Opportunistic power scheduling for dynamic multiserver wireless
  systems.
\newblock {\em IEEE Transactions on Wireless Communications}, vol. 5, no.6, pp.
  1506-1515, June 2006.

\bibitem{lin-shroff-cdc04}
X.~Lin and N.~B. Shroff.
\newblock Joint rate control and scheduling in multihop wireless networks.
\newblock {\em Proc. of 43rd IEEE Conf. on Decision and Control, Paradise
  Island, Bahamas}, Dec. 2004.

\bibitem{vijay-allerton02}
R.~Agrawal and V.~Subramanian.
\newblock Optimality of certain channel aware scheduling policies.
\newblock {\em Proc. 40th Annual Allerton Conference on Communication ,
  Control, and Computing, Monticello, IL}, Oct. 2002.

\bibitem{prop-fair-down}
H.~Kushner and P.~Whiting.
\newblock Asymptotic properties of proportional-fair sharing algorithms.
\newblock {\em Proc. of 40th Annual Allerton Conf. on Communication, Control,
  and Computing}, 2002.

\bibitem{neely-non-convex}
M.~J. Neely.
\newblock Stochastic network optimization with non-convex utilities and costs.
\newblock {\em Proc. Information Theory and Applications Workshop (ITA)}, Feb.
  2010.

\bibitem{olav-martingale}
O.~Kallenberg.
\newblock {\em Foundations of Modern Probability, 2nd ed., Probability and its
  Applications}.
\newblock Springer-Verlag, 2002.

\bibitem{mart-approx-filtration}
Y.~V. Borovskikh and V.~S. Korolyuk.
\newblock {\em Martingale Approximation}.
\newblock VSP BV, The Netherlands, 1997.

\bibitem{longbo-lagrange}
L.~Huang and M.~J. Neely.
\newblock Delay reduction via lagrange multipliers in stochastic network
  optimization.
\newblock {\em Proc. of 7th Intl. Symposium on Modeling and Optimization in
  Mobile, Ad Hoc, and Wireless Networks (WiOpt)}, June 2009.

\bibitem{bertsekas-nonlinear}
D.~P. Bertsekas.
\newblock {\em Nonlinear Programming}.
\newblock Athena Scientific, Belmont, MA, 1995.

\bibitem{sno-text}
M.~J. Neely.
\newblock {\em Stochastic Network Optimization with Application to
  Communication and Queueing Systems}.
\newblock Morgan \& Claypool, 2010.

\end{thebibliography}
\end{document}